\documentclass[a4paper]{article}
\usepackage{amsmath,amsfonts,amssymb,amsthm,pstricks,caption}

\newtheorem{lemma}{Lemma}[section]
\newtheorem{corollary}[lemma]{Corollary}
\newtheorem{theorem}[lemma]{Theorem}
\newtheorem{proposition}[lemma]{Proposition}

\theoremstyle{definition}
\newtheorem{definition}[lemma]{Definition}
\newtheorem{remark}[lemma]{Remark}
\newtheorem{example}[lemma]{Example}

\newcommand{\Sphere}{\mathbb{S}^3}  
\DeclareMathOperator*{\nhd}{\mathcal{N}} 
\DeclareMathOperator{\inc}{i}
\DeclareMathOperator*{\ms}{MS}
\DeclareMathOperator{\is}{IS}
\DeclareMathOperator{\V}{V}
\DeclareMathOperator{\lk}{lk}

\DeclareMathOperator{\Int}{int}

\begin{document}
\title{The Kakimizu complex of a split link}
\author{Jessica E. Banks}
\date{}
\maketitle
\begin{abstract}
We study the Kakimizu complex of a split link. As part of this, we also study Seifert surfaces and the Kakimizu complex for a non-split link in a 3--ball.
In addition, we show that a simplex of the Kakimizu complex of a non-split link can be realised in an essentially unique way.
\end{abstract}


\section{Introduction}

The Kakimizu complex of a link records the structure of the set of taut Seifert surfaces for the link.
To date, research on Seifert surfaces and the Kakimizu complex has focused on non-split links. 
Perhaps the main reason for this is that the complement of a split link is reducible, and many 3--manifold techniques are best suited to working with irreducible manifolds.

Our intention here is to study the Kakimizu complex for split links. In order to do so, we must first understand Seifert surfaces in a link complement within a 3--ball, or equivalently in the complement of a point in $\Sphere$, or in $\mathbb{R}^3$.
While some of the results below are unsurprising, our aim is to give complete proofs.

\bigskip

Let $L\subset\Sphere$ be an oriented link, and set $M=\Sphere\setminus\nhd(L)$.
Let $\pi\colon\widetilde{M}\to M$ be the infinite cyclic cover of $M$, corresponding to the kernel of the homomorphism $\lk\colon\pi_1(M)\to\mathbb{Z}$ given by linking number with $L$. Let $\tau\colon\widetilde{M}\to\widetilde{M}$ be the generating covering transformation in the positive direction (that is, $\tau$ corresponds to an element of $\pi_1(M)$ that has linking number 1 with $L$).

\begin{definition}
A \textit{Seifert surface} for $L$ is a compact, oriented surface $R\subset\Sphere$, with no closed components, such that $\partial R=L$ as an oriented link. We study Seifert surfaces up to isotopy of $\Sphere$ keeping $L$ fixed.
The surface $R$ can also be seen as properly embedded in $M$, considered up to ambient isotopy in $M$. 

We say $R$ is \textit{taut} if it has maximal Euler characteristic among all Seifert surfaces for $L$.
\end{definition}

\begin{definition}[See \cite{MR1177053} p225  and \cite{MR2869183} p1490]\label{msldefn}
Define the \textit{Kakimizu complex} $\ms(L)$ of $L$ to be the following flag simplicial complex. Its vertices are ambient isotopy classes of taut Seifert surfaces for $L$. Two distinct vertices span an edge if they have representatives $R$, $R'$ such that a lift of $M \setminus R'$ to $\widetilde{M}$ intersects exactly two lifts of $M \setminus R$.
\end{definition}

Note that two distinct taut Seifert surfaces that are adjacent can be made disjoint.
The converse is true if all taut Seifert surfaces for $L$ are connected, but does not necessarily hold if there are disconnected taut Seifert surfaces. 
This is why the definition uses lifts of $M\setminus R$ rather than lifts of $R$.
For succinctness, we will take `a lift of a Seifert surface' to mean a connected lift. 
That is, we will say $\widetilde{R}$ is a lift of a taut Seifert surface $R$ if there is a lift $V$ of $M\setminus R$ such that $\widetilde{R}$ lies between $V$ and $\tau(V)$.

\begin{definition}
Say taut Seifert surfaces $R$ and $R'$ with $[R]$ and $[R']$ adjacent are \textit{tight} if they realise the adjacency.
\end{definition}

\begin{remark}
When Kakimizu first defined $\ms(L)$ in \cite{MR1177053}, he also defined a second complex $\is(L)$ that is constructed using incompressible Seifert surfaces for $L$ rather than just taut ones. All the results in this paper also hold in this setting.
\end{remark}

We will make use of the following definitions and results to understand isotopies of Seifert surfaces.

\begin{definition}[\cite{MR2869183} Section 3] \label{almostdefn}
Let $M'$ be a connected 3--manifold, and let $S,S'$ be (possibly disconnected) surfaces properly embedded in $M'$.
Call $S$ and $S'$ \textit{almost transverse} if, given a component $S_0$ of $S$ and a component $S'_0$ of $S'$, they either coincide or intersect transversely. Call the surfaces \textit{almost disjoint} if, given a component $S_0$ of $S$ and a component $S'_0$ of $S$, they either coincide or are disjoint. Say they are \textit{$\partial$--almost disjoint} if $\partial S = \partial S'$ and, given a component $S_0$ of $S$ and a component $S'_0$ of $S'$, they either coincide or have disjoint interiors.
Say $S$ and $S'$ bound a \textit{product region} if the following holds. There is a compact surface $S_T$, a finite collection $\rho_T\subseteq\partial R$ of arcs and simple closed curves and a map of $T = (S_T \times I)\mathclose{}/\mathopen{}\sim$ into $M$ that is an embedding on the interior of $T$ and has the following properties.
\begin{itemize}
\item $S_T \times \{0\} = S \cap T$ and $S_T \times \{1\} = S' \cap T$.
\item $\partial T\setminus(S_T\times \partial I)\subseteq \partial M$.
\end{itemize}
Here $\sim$ collapses $\{x\}\times I$ to a point for each $x \in \rho_T$. 
The \textit{horizontal boundary} of $T$ is $(S_T \times\partial I)/\mathopen{}\sim\mathclose{}$. 
Say $S$ and $S'$ have \textit{simplified intersection} if they do not bound a product region.
\end{definition}

\begin{proposition}[\cite{MR0224099} Corollary 3.2]\label{surfaceinproductprop}
Suppose surfaces $S_0,S_1$ bound a product region $T=(S_T\times I)/\mathopen{}\sim\mathclose{}$. Let $S'$ be an incompressible surface that is transverse to $S_0,S_1$. Suppose $S'\cap \Int(T)\neq\emptyset$ but $S'\cap (S_1\cap T)=\emptyset$. Then each component of $S'\cap\Int(T)$ bounds a product region in $T$ with a subsurface of $S_0$. In particular, if additionally $S'\cap (S_0\cap T)=\emptyset$ then this component of $S'$ is parallel to those of $S_0,S_1$ that bound $T$.
\end{proposition}

\begin{theorem}[\cite{MR1315011} Proposition 4.8; see also \cite{MR0224099} Proposition 5.4 and Corollary 3.2]\label{productregionthm}
Let $M'$ be an irreducible, $\partial$--irreducible Haken manifold. Let $S, S'$ be incompressible, $\partial$--incompressible
surfaces properly embedded in $M'$. 
\begin{itemize}
 \item If $S$ and $S'$ are isotopic then there is a product region between them in $M'$.
 \item Suppose $S \cap S' \neq\emptyset$, but $S$ can be isotoped to be disjoint from $S'$. Then there is a
product region between $S$ and $S'$ in $M'$.
\end{itemize}
\end{theorem}

If $S$ is an incompressible Seifert surface then it is $\partial$--incompressible.
If $M'$ is the complement of a non-split link other than the unknot, or the infinite cyclic cover of such a manifold, then $M'$ is irreducible, $\partial$--irreducible and Haken. However, for much of this paper we will be working with manifolds that do not meet these hypotheses.

\begin{corollary}[\cite{2011arXiv1109.0965B} Corollary 4.5]\label{isotopyrelbdycor1}
Suppose $L$ is not split and not fibred. Let $R,R'$ be taut Seifert surfaces for $L$. If $R,R'$ do not coincide but are isotopic by an isotopy fixing their boundaries then there is a product region $T=(S_T\times I)/\mathopen{}\sim\mathclose{}$ between them with $\rho_T=\partial T$. 
\end{corollary}

\begin{corollary}\label{isotopyrelbdycor2}
Suppose $L$ is not split and not fibred. Let $R,R'$ be taut Seifert surfaces for $L$. If $(R\cap R')\setminus\partial R\neq\emptyset$ but $R'$ can be isotoped keeping its boundary fixed so that $R$ and $R'$ are $\partial$--almost disjoint then there is a product region $T=(S_T\times I)/\mathopen{}\sim\mathclose{}$ between them with $\rho_T=\partial T$. 
\end{corollary}
\begin{proof}
Let $R''$ be a copy of $R'$, and isotope it keeping its boundary fixed so that $R$ and $R''$ are $\partial$--almost disjoint. Keeping it $\partial$--almost disjoint from $R$, isotope $R''$ to minimise $|R'\cap R''|$.
By Corollary \ref{isotopyrelbdycor1}, there is a product region $T'$ between $R'$ and $R''$ with $\rho_{T'}=\partial T'$.
If $R$ is disjoint from $T'$ then $T'$ gives an isotopy of $R''$ keeping its boundary fixed and disjoint from $R$ that reduces $|R'\cap R''|$. No such isotopy exists. Thus $R'$ meets $T'$.
Since $R$ is $\partial$--almost disjoint from $R''$, Proposition \ref{surfaceinproductprop} tells us that each component of $R\cap T'$ bounds a product region in $T'$ with $R'$.
These product regions are partially ordered by inclusion. Let $T$ be one such product region that is minimal  in this order. Then $T$ is as required.
\end{proof}

\bigskip

In Section \ref{complementofpointsection} we study the Kakimizu complex of a non-split link in $\mathbb{R}^3$. We will see that the vertices can be identified in terms of the Kakimizu complex of the link and the fundamental group of the link complement. We then describe when two vertices are adjacent. One result we obtain is as follows (see Definition \ref{mspldefn}).

\begin{theorem}
Let $L$ be a link such that every taut Seifert surface for $L$ is connected. Then $\dim(\ms_p(L))=\dim(\ms(L))+1$.
\end{theorem}

In Section \ref{splitlinksection} we consider how to understand Seifert surfaces for a split link in terms of those for the components that make up the link. 
Describing the behaviour in this situation explicitly and in full generality does not appear straightforward. However, for the distant union of two non-split links we have the following more precise result.

\begin{theorem}
Let $L_1$ and $L_2$ be non-split links. Then $|\ms(L_1\sqcup L_2)|\cong|\ms_p(L_1)|\times|\ms_p(L_2)|$.
\end{theorem}

In the appendix we prove some results we need in Section \ref{complementofpointsection} about positioning surfaces representing the vertices of a simplex in $\ms(L)$. These give the following.

\begin{theorem}\label{uniquepositionthm}
Suppose $L$ is not split. Let $[R_0],\ldots,[R_n]$ be the distinct vertices of a simplex of $\ms(L)$. 
Then $R_0,\ldots,R_n$ can be isotoped to be pairwise disjoint and tight.
Furthermore, this position is unique up to ambient isotopy of $R_0\cup\cdots\cup R_n$ in $M$.
\end{theorem}

The material in the appendix is not used in Section \ref{splitlinksection}, and in Section \ref{complementofpointsection} it is only used in the proof of Theorem \ref{realisesimplexthm}.

\bigskip

I began to consider this subject while visiting the University of California, Davis, supported by the Cecil King Travel Scholarship 2011. I wish to thank the London Mathematical Society and the Cecil King Foundation for their support, and UC Davis for their hospitality.
I am also grateful to Jennifer Schultens for her interest in this project. 


\section{Seifert surfaces in $\mathbb{R}^3$}\label{complementofpointsection}
\subsection{Identifying surfaces}

Assume for the whole of this section that $L$ is non-split.
Fix a point $p\in M$, and let $M_p=M\setminus\{p\}$.
Let $\inc\colon M_p\to M$ be the inclusion map.
Let $\widetilde{M}_p=\widetilde{M}\setminus\pi^{-1}(p)$. Then $\widetilde{M}_p$ is the infinite cyclic cover of $M_p$.

We can define the Kakimizu complex of $L$ in $M_p$.

\begin{definition}\label{mspldefn}
Define $\ms_p(L)$ to be the following flag simplicial complex. Its vertices are ambient isotopy classes of taut Seifert surfaces for $L$ in $M_p$. Two distinct vertices span an edge if they have representatives $R,R'$ such that a lift of $M_p \setminus R'$ to $\widetilde{M}_p$ intersects exactly two lifts of $M_p \setminus R$.
\end{definition}

Let $R$ be a taut Seifert surface for $L$ in $M_p$. Then $\inc(R)$ is a taut Seifert surface for $L$ in $M$. 
In addition, isotoping $R$ in $M_p$ only changes $\inc(R)$ by an isotopy in $M$.
Let $R'$ be another taut Seifert surface in $M_p$ such that $[R]$ and $[R']$ are adjacent in $\ms_p(L)$, and $R$ and $R'$ are tight.
That is, suppose a lift of $M_p \setminus R'$ to $\widetilde{M}_p$ intersects exactly two lifts of $M_p \setminus R$.
Then a lift of $M \setminus \inc(R')$ to $\widetilde{M}$ intersects exactly two lifts of $M \setminus \inc(R)$.
Thus $[\inc(R)]$ and $[\inc(R')]$ are at most distance 1 apart in $\ms(L)$. 

This shows that $\inc$ induces a simplicial map $\inc_*\colon\ms_p(L)\to\ms(L)$.
In general, $\inc_*$ will not be injective. On the other hand, it is clearly surjective, as any Seifert surface in $M$ can be made disjoint from $p$.
We will write $\inc_*(R)$ for $\inc_*([R])$.

\begin{lemma}\label{fibredgivesbijectionlemma}
Suppose $L$ is fibred (as an oriented link in $\Sphere$). Then $\inc_*$ is a bijection.
\end{lemma}
\begin{proof}
Let $R$ be a taut Seifert surface for $L$.
Since $L$ is fibred, the closure of $M\setminus R$ is of the form $R\times[0,1]$.
Isotoping $R$ from $R\times\{0\}$ to $R\times\{1\}$ will not in general preserve $R$ pointwise but will preserve it setwise.
Suppose that $p$ lies on $R\times\{1-\varepsilon\}$ for some small $\varepsilon>0$.
Then there is a copy $R'$ of $R$ that is isotopic to $R$ in $M_p$ that coincides with $R$ outside the $2\varepsilon$--neighbourhood $\nhd(p)$ of $p$ such that within this neighbourhood $R$ and $R'$ are as shown in Figure \ref{pic1}.
\begin{figure}[htbp]
\centering
\psset{xunit=.75pt,yunit=.75pt,runit=.75pt}
\begin{pspicture}(130,110)
{
\pscustom[linewidth=1,linecolor=black,linestyle=dashed,dash=4 4,fillstyle=solid,fillcolor=lightgray]
{
\newpath
\moveto(115,54.99999738)
\curveto(115,27.38575989)(92.61423749,4.99999738)(65,4.99999738)
\curveto(37.38576251,4.99999738)(15,27.38575989)(15,54.99999738)
\curveto(15,82.61423487)(37.38576251,104.99999738)(65,104.99999738)
\curveto(92.61423749,104.99999738)(115,82.61423487)(115,54.99999738)
\closepath
}
}
{
\pscustom[linestyle=none,fillstyle=solid,fillcolor=black]
{
\newpath
\moveto(70,54.99999738)
\curveto(70,52.23857363)(67.76142375,49.99999738)(65,49.99999738)
\curveto(62.23857625,49.99999738)(60,52.23857363)(60,54.99999738)
\curveto(60,57.76142113)(62.23857625,59.99999738)(65,59.99999738)
\curveto(67.76142375,59.99999738)(70,57.76142113)(70,54.99999738)
\closepath
}
}
{
\pscustom[linewidth=1,linecolor=black]
{
\newpath
\moveto(5,80)
\lineto(125,80)
}
}
{
\pscustom[linewidth=1,linecolor=black]
{
\newpath
\moveto(5,80)
\lineto(25,80)
\curveto(35,80)(45,30)(65,30)
\curveto(85,30)(95,80)(105,80)
\lineto(125,80)
}
}
{
\put(70,60){$p$}
\put(70,82){$R$}
\put(85,37){$R'$}
\put(40,15){$\nhd(p)$}
}
\end{pspicture}
\caption{\label{pic1}}
\end{figure}
From this we see that any isotopy of $R$ in $M$ can be changed to an isotopy of $R$ in $M_p$.
Hence $\inc_*$ is injective.
\end{proof}

We now assume for the rest of this section that $L$ is not fibred. In particular this means that $L$ is not the unknot.

\begin{lemma}\label{fixboundarylemma}
Let $R,R'$ be Seifert surfaces for $L$ that are isotopic in $M$ by an isotopy keeping the boundary fixed. 
Suppose also that they are isotopic in $M_p$. Then the isotopy in $M_p$ can be made to fix the boundary.
\end{lemma}
\begin{proof}
Let $\widetilde{R}$ be a lift of $R$ to $\widetilde{M}_p$. There is an isotopy of $\widetilde{R}$ in $\widetilde{M}_p$ to a lift $\widetilde{R}'$ of $R'$. 
Then $\partial \widetilde{R}'=\tau^m(\partial\widetilde{R})$ for some $m\in\mathbb{Z}$. 
In $\widetilde{M}$, there is also an isotopy from $\widetilde{R}'$ to $\tau^m(\widetilde{R})$.
Hence $\widetilde{R}$ is isotopic to $\tau^m(\widetilde{R})$ in $\widetilde{M}$.
As $L$ is not fibred, this implies that $m=0$.
\end{proof}

\begin{definition}\label{homeodefn}
Given a directed simple closed curve $\rho$ based at $p$, define a homeomorphism from $M$ to itself by isotoping $p$ once around $\rho$. Let $\phi_{\rho}$ be the restriction of this homeomorphism to $M_p$.  
\end{definition}

Here we want to think of $\rho$ as a directed path with both its endpoints at $p$. We will drop the term `directed'.

\begin{lemma}\label{pathexistslemma}
Let $R$ and $R'$ be taut Seifert surfaces for $L$ in $M_p$ such that $\inc_*(R)=\inc_*(R')$.
Then there exists a simple closed curve $\rho$ based at $p$ such that $[R']=[\phi_{\rho}(R)]$.
\end{lemma}
\begin{proof}
Let $\rho$ be the path of $p$ under the ambient isotopy of $M$ taking $\inc(R)$ to $\inc(R')$. 
Since $R$ and $R'$ are disjoint from $p$ and $R'$ is non-separating, we may ensure that $\rho$ has both its endpoints at $p$.
\end{proof}

Fix a taut Seifert surface $R$ for $L$, and a simple closed curve $\rho$ based at $p$.

\begin{lemma}\label{homotopelooplemma}
Changing $\rho$ to another simple closed curve within its homotopy class relative to its endpoints does not change the isotopy class of $\phi_{\rho}(R)$.
\end{lemma}
\begin{proof}
Suppose first that we isotope a point of $\rho$ across $R$, creating a new component of $\rho\setminus R$ that is contained within a small neighbourhood of a point of $R$.
Up to isotopy, this does not alter the surface $\phi_{\rho}(R)$, as can be seen from Figure \ref{pic2}.
\begin{figure}[htb]
\centering
\input{pic2}
\caption{\label{pic2}}
\end{figure}
From this we see that
isotoping $\rho$ relative to its endpoints does not change the isotopy class of $\phi_{\rho}(R)$.

In fact, when isotoping $\rho$, we can pull the endpoints of $\rho$ apart slightly, perform an isotopy, and then return the endpoints to the same point. Doing this will not change the resulting surface $\phi_{\rho}(R)$.
Knowing this, it follows that we may actually change $\rho$ to any other simple loop in its homotopy class relative to its endpoints. To see this, note that we may put any such homotopy into general position, so that the homotopy is only not an isotopy at finitely many points, where a crossing change takes place. We can replace any crossing change by an isotopy that sends one section of the arc around the end of the arc.
\end{proof}

\begin{remark}\label{actionremark}
Since $[\phi_{\sigma}(\phi_{\rho}(R))]=[\phi_{\rho\cdot\sigma}(R)]$ we see that $\pi_1(M;p)$ acts on $\ms_p(L)$.
\end{remark}

We may therefore assume that $\rho$ has been homotoped to minimise its intersection with $R$ while keeping its endpoints fixed. 
Note that $\phi_{\rho}(R)$ coincides with $R$ outside a neighbourhood of the points of $R\cap\rho$.

\begin{remark}\label{disjointpathremark}
If $\rho$ is disjoint from $R$ then $\phi_{\rho}(R)=R$. 
\end{remark}

In the converse direction, we have the following.

\begin{lemma}\label{pathdisjointlemma}
If $\phi_{\rho}(R)$ is isotopic to $R$ in $M_p$ then $\rho$ is disjoint from $R$.
\end{lemma}
\begin{proof}
Assume that $\phi_{\rho}(R)$ is isotopic to $R$ but $\rho$ is not disjoint from $R$.
Without loss of generality, assume that $\rho$ first crosses $R$ in the positive direction (the direction corresponding to $\tau$). 

Let $x$ be the point on $R$ where $\rho$ first meets it, and let $\sigma$ be a path in $R$ that is away from any other points of $\rho$, running from $x$ to a point on $\partial R$.
Push $\sigma$ downwards off $R$, so that $x$ now lies on $\rho$ immediately before it reaches $R$, and the other end of $\sigma$ still lies on $\partial M$.
Isotope $\rho$ so that it runs to $x$, runs along $\sigma$ and back again, and then continues as before.
This does not change $\rho\cap R$.
Let $\sigma'$ be the part of $\rho$ from where is meets $\partial M$. Then $\sigma'$ is disjoint from $\phi_{\rho}(R)$ but not from $R$.

By Lemma \ref{fixboundarylemma}, we may isotope a copy $R'$ of
$\phi_{\rho}(R)$ to $R$ in $M_p$ keeping its boundary fixed.
The ambient isotopy moves $\sigma'$, but keeps its endpoints fixed. After the isotopy, $\sigma'$ is disjoint from $R$. This contradicts that $\rho$ was chosen with $\rho\cap R$ minimal.
\end{proof}

Now assume that $[\phi_{\rho}(R)]$ is distinct from $[R]$. That is, suppose that $\rho$ is not (and cannot be homotoped to be) disjoint from $R$.
Consider a second path $\rho'$. 
Suppose that $\phi_{\rho'}(R)$ is also not isotopic to $R$. When is $\phi_{\rho'}(R)$ isotopic to $\phi_{\rho}(R)$?

\begin{lemma}\label{pathcoincidelemma}
Suppose $\phi_{\rho}(R)$ and $\phi_{\rho'}(R)$ are isotopic in $M_p$.
Then $\rho'$ can be homotoped relative to its endpoints so that the following holds.
The paths $\rho$ and $\rho'$ first meet $R$ at the same point, after which they coincide.
That is, there exists a simple closed curve $\sigma$ based at $p$ and disjoint from $R$ such that $[\rho']=[\sigma\cdot\rho]$.
\end{lemma}
\begin{proof}
Let $y$ be a point on $\rho'$ just before it first meets $R$. As in the proof of Lemma \ref{pathdisjointlemma}, isotope $\rho'$ so that it runs to $y$, then follows a path immediately below $R\setminus(\rho\cup\rho')$ to $\partial M$, returns along this path to $y$, and finally continues along its original path to $p$.
Let $\sigma_y$ be the section of $\rho'$ from where it meets $\partial M$.
Then $\sigma_y$ is disjoint from $\phi_{\rho'}(R)$. 

Again Lemma \ref{fixboundarylemma} shows that the isotopy between $\phi_{\rho}(R)$ and $\phi_{\rho'}(R)$ can be made to fix the boundary.
Take a copy $R_{\rho'}$ of $\phi_{\rho'}(R)$, and isotope it to coincide with $\phi_{\rho}(R)$, moving (the interior of) $\sigma_y$ as needed in the process. Then $\sigma_y$ is disjoint from $\phi_{\rho}(R)$.

Let $x_1,\ldots,x_k$ be the points of $\rho\cap R$, in order as measured along $\rho$.
In a neighbourhood of $x_1$ there is a disc $D_0$ in $R$ that is properly embedded in the complement of $\phi_{\rho}(R)$. The boundary of this disc also bounds a disc $D_1$ in $\phi_{\rho}(R)$. The sphere $D_0\cup D_1$ bounds a 3--ball $B_p$ in $M$ that contains $p$.
Since $\sigma_y(1)=p$ but $\sigma_y(0)$ lies outside $B_p$, the arc $\sigma_y$ must cross $D_0$ an odd number of times.

Suppose $|D_0\cap \sigma_y|\geq 3$. Choose an arc $\sigma_p$ in $B_p$ from $p$ to $D_1$ that is disjoint from $\sigma_y$ except at $p$.
The boundary of a regular neighbourhood of $\sigma_p$ in $B_p$ is a disc $D_2$ properly embedded in $B_p$, and $|D_2\cap\sigma_y|=1$. 
By isotoping $D_2$ through $B_p$ to $D_0$ we can find an isotopy of $\sigma_y$ that reduces $|D_0\cap\sigma_y|$ to 1.
We may then isotope $\sigma_y$ to coincide with $\rho$ within $B_p$. 

Suppose that there is a point $y_0$ of $\sigma_y\cap R$ that occurs earlier on $\sigma_y$. 
We may assume $y_0$ is the first such point as measured along $\rho'$.
Since $\sigma_y$ is disjoint from $\phi_{\rho}(R)$, there is some $i\leq k$ such that $y_0$ lies in a neighbourhood of $x_i$.
We have already seen that $i\neq 1$.
Therefore there is an annulus $A$ in $R$, contained in a neighbourhood of $x_i$, that is properly embedded in the complement of $\phi_{\rho}(R)$ and contains $y_0$.
The two boundary components of $A$ each bound a disc in $\phi_{\rho}(R)$. Call these discs $D_3$ and $D_4$. Figure \ref{pic3} shows this schematically.
The sphere $A\cup D_3\cup D_4$ bounds a 3--ball in $M$ with interior disjoint from $B_p$.
Thus there is an isotopy that pushes $\sigma_y$ out of this 3--ball and reduces $|\sigma_y\cap R|$.
\begin{figure}[htbp]
\centering
\psset{xunit=.25pt,yunit=.25pt,runit=.25pt}
\begin{pspicture}(1223,763)
{
\newgray{lightgrey}{0.85}
}
{
\pscustom[linewidth=2,linecolor=black,fillstyle=solid,fillcolor=lightgrey]
{
\newpath
\moveto(260,500)
\lineto(20,220)
\lineto(1040,220)
\lineto(1200,500)
\lineto(260,500)
\closepath
}
}
{
\pscustom[linestyle=none,fillstyle=solid,fillcolor=lightgray]
{
\newpath
\moveto(681.5,351.49999738)
\curveto(681.5,323.88575989)(596.43410247,301.49999738)(491.5,301.49999738)
\curveto(386.56589753,301.49999738)(301.5,323.88575989)(301.5,351.49999738)
\curveto(301.5,379.11423487)(386.56589753,401.49999738)(491.5,401.49999738)
\curveto(596.43410247,401.49999738)(681.5,379.11423487)(681.5,351.49999738)
\closepath
}
}
{
\pscustom[linestyle=none,fillstyle=solid,fillcolor=lightgrey]
{
\newpath
\moveto(560,351.49999738)
\curveto(560,334.93145489)(528.65993249,321.49999738)(490,321.49999738)
\curveto(451.34006751,321.49999738)(420,334.93145489)(420,351.49999738)
\curveto(420,368.06853988)(451.34006751,381.49999738)(490,381.49999738)
\curveto(528.65993249,381.49999738)(560,368.06853988)(560,351.49999738)
\closepath
}
}
{
\pscustom[linewidth=2,linecolor=black]
{
\newpath
\moveto(301.5,351.5)
\lineto(301.5,691.5)
\curveto(301.5,691.5)(302.77178,741.49999)(491.5,741.49999)
\curveto(681.5,741.49999)(681.5,691.5)(681.5,691.5)
\lineto(681.5,351.5)
\moveto(420,220)
\lineto(420,100)
\curveto(420,100)(420.13165,21.5)(701.5,21.5)
\curveto(981.5,21.5)(980,100)(980,100)
\lineto(980,220)
\moveto(560,220)
\curveto(560,220)(557.86393,161.5)(701.5,161.5)
\curveto(841.5,161.5)(840,220)(840,220)
}
}
{
\pscustom[linewidth=2,linecolor=black,linestyle=dashed,dash=3 3]
{
\newpath
\moveto(301.5,351.4999997)
\curveto(301.50000487,379.11423719)(386.56590635,401.49999866)(491.50000882,401.49999738)
\curveto(596.2898997,401.4999961)(681.29593932,379.17342765)(681.49964102,351.59719268)
\moveto(420,351.49999878)
\curveto(420.00000179,368.06854127)(451.34007076,381.49999815)(490.00000325,381.49999738)
\curveto(528.65993574,381.49999661)(560.00000179,368.06853849)(560,351.49999599)
\curveto(560,351.46390352)(559.99984801,351.42781107)(559.99954406,351.39171883)
}
}
{
\pscustom[linewidth=2,linecolor=black]
{
\newpath
\moveto(840,351.49999878)
\curveto(840.00000179,368.06854127)(871.34007076,381.49999815)(910.00000325,381.49999738)
\curveto(948.55192467,381.49999662)(979.8470691,368.14070126)(979.99945316,351.61857829)
}
}
{
\pscustom[linewidth=3,linecolor=gray]
{
\newpath
\moveto(341.5,141.5)
\lineto(361.5,581.5)
\moveto(491.5,571.5)
\lineto(491.5,351.5)
\lineto(491.5,101.5)
\lineto(911.5,101.5)
\lineto(911.5,561.5)
}
}
{
\pscustom[linestyle=none,fillstyle=solid,fillcolor=gray]
{
\newpath
\moveto(340.24674856,113.92846833)
\curveto(341.5708648,113.86828122)(342.59548181,112.74608164)(342.5352947,111.42196541)
\curveto(342.4751076,110.09784917)(341.35290802,109.07323216)(340.02879178,109.13341926)
\curveto(338.70467555,109.19360636)(337.68005854,110.31580595)(337.74024564,111.63992218)
\curveto(337.80043274,112.96403842)(338.92263232,113.98865543)(340.24674856,113.92846833)
\closepath
\moveto(340.723529,124.417638)
\curveto(342.04764524,124.3574509)(343.07226225,123.23525131)(343.01207514,121.91113508)
\curveto(342.95188804,120.58701884)(341.82968846,119.56240183)(340.50557222,119.62258893)
\curveto(339.18145599,119.68277604)(338.15683898,120.80497562)(338.21702608,122.12909185)
\curveto(338.27721318,123.45320809)(339.39941276,124.4778251)(340.723529,124.417638)
\closepath
\moveto(341.20030944,134.90680767)
\curveto(342.52442568,134.84662057)(343.54904269,133.72442099)(343.48885558,132.40030475)
\curveto(343.42866848,131.07618851)(342.3064689,130.0515715)(340.98235266,130.11175861)
\curveto(339.65823643,130.17194571)(338.63361942,131.29414529)(338.69380652,132.61826153)
\curveto(338.75399362,133.94237776)(339.8761932,134.96699477)(341.20030944,134.90680767)
\closepath
\moveto(362.01764734,592.88824139)
\curveto(363.34176357,592.82805429)(364.36638058,591.70585471)(364.30619348,590.38173847)
\curveto(364.24600638,589.05762224)(363.1238068,588.03300523)(361.79969056,588.09319233)
\curveto(360.47557432,588.15337943)(359.45095731,589.27557901)(359.51114442,590.59969525)
\curveto(359.57133152,591.92381149)(360.6935311,592.9484285)(362.01764734,592.88824139)
\closepath
\moveto(362.49442778,603.37741107)
\curveto(363.81854401,603.31722396)(364.84316102,602.19502438)(364.78297392,600.87090815)
\curveto(364.72278682,599.54679191)(363.60058724,598.5221749)(362.276471,598.582362)
\curveto(360.95235476,598.6425491)(359.92773775,599.76474869)(359.98792486,601.08886492)
\curveto(360.04811196,602.41298116)(361.17031154,603.43759817)(362.49442778,603.37741107)
\closepath
\moveto(362.97120822,613.86658074)
\curveto(364.29532445,613.80639364)(365.31994146,612.68419405)(365.25975436,611.36007782)
\curveto(365.19956726,610.03596158)(364.07736768,609.01134457)(362.75325144,609.07153167)
\curveto(361.4291352,609.13171878)(360.40451819,610.25391836)(360.4647053,611.57803459)
\curveto(360.5248924,612.90215083)(361.64709198,613.92676784)(362.97120822,613.86658074)
\closepath
\moveto(911.5,572.90000004)
\curveto(912.82548342,572.90000004)(913.90000004,571.82548342)(913.90000004,570.5)
\curveto(913.90000004,569.17451658)(912.82548342,568.09999996)(911.5,568.09999996)
\curveto(910.17451658,568.09999996)(909.09999996,569.17451658)(909.09999996,570.5)
\curveto(909.09999996,571.82548342)(910.17451658,572.90000004)(911.5,572.90000004)
\closepath
\moveto(911.5,583.40000004)
\curveto(912.82548342,583.40000004)(913.90000004,582.32548342)(913.90000004,581)
\curveto(913.90000004,579.67451658)(912.82548342,578.59999996)(911.5,578.59999996)
\curveto(910.17451658,578.59999996)(909.09999996,579.67451658)(909.09999996,581)
\curveto(909.09999996,582.32548342)(910.17451658,583.40000004)(911.5,583.40000004)
\closepath
\moveto(911.5,593.90000004)
\curveto(912.82548342,593.90000004)(913.90000004,592.82548342)(913.90000004,591.5)
\curveto(913.90000004,590.17451658)(912.82548342,589.09999996)(911.5,589.09999996)
\curveto(910.17451658,589.09999996)(909.09999996,590.17451658)(909.09999996,591.5)
\curveto(909.09999996,592.82548342)(910.17451658,593.90000004)(911.5,593.90000004)
\closepath
}
}
{
\pscustom[linestyle=none,fillstyle=solid,fillcolor=black]
{
\newpath
\moveto(496.5,351.49999738)
\curveto(496.5,348.73857363)(494.26142375,346.49999738)(491.5,346.49999738)
\curveto(488.73857625,346.49999738)(486.5,348.73857363)(486.5,351.49999738)
\curveto(486.5,354.26142113)(488.73857625,356.49999738)(491.5,356.49999738)
\curveto(494.26142375,356.49999738)(496.5,354.26142113)(496.5,351.49999738)
\closepath
\moveto(496.5,571.49999738)
\curveto(496.5,568.73857363)(494.26142375,566.49999738)(491.5,566.49999738)
\curveto(488.73857625,566.49999738)(486.5,568.73857363)(486.5,571.49999738)
\curveto(486.5,574.26142113)(488.73857625,576.49999738)(491.5,576.49999738)
\curveto(494.26142375,576.49999738)(496.5,574.26142113)(496.5,571.49999738)
\closepath
\moveto(356.5,351.49999738)
\curveto(356.5,348.73857363)(354.26142375,346.49999738)(351.5,346.49999738)
\curveto(348.73857625,346.49999738)(346.5,348.73857363)(346.5,351.49999738)
\curveto(346.5,354.26142113)(348.73857625,356.49999738)(351.5,356.49999738)
\curveto(354.26142375,356.49999738)(356.5,354.26142113)(356.5,351.49999738)
\closepath
\moveto(916.5,351.49999738)
\curveto(916.5,348.73857363)(914.26142375,346.49999738)(911.5,346.49999738)
\curveto(908.73857625,346.49999738)(906.5,348.73857363)(906.5,351.49999738)
\curveto(906.5,354.26142113)(908.73857625,356.49999738)(911.5,356.49999738)
\curveto(914.26142375,356.49999738)(916.5,354.26142113)(916.5,351.49999738)
\closepath
}
}
{
\pscustom[linewidth=2,linecolor=black,linestyle=dashed,dash=3 3]
{
\newpath
\moveto(840,350)
\lineto(840,220)
\moveto(980,350)
\lineto(980,220)
\moveto(420,220)
\lineto(421.5,621.5)
\curveto(421.5,621.5)(422.75656,641.5)(491.5,641.5)
\curveto(561.5,641.5)(561.5,621.5)(561.5,621.5)
\lineto(560,220)
}
}
{
\pscustom[linewidth=2,linecolor=black]
{
\newpath
\moveto(20,220)
\lineto(1040,220)
}
}
{
\pscustom[linewidth=2,linecolor=black]
{
\newpath
\moveto(980,351.49999738)
\curveto(980,334.93145489)(948.65993249,321.49999738)(910,321.49999738)
\curveto(871.34006751,321.49999738)(840,334.93145489)(840,351.49999738)
\curveto(840,351.63195533)(840.0020315,351.76391183)(840.00609438,351.89585829)
\moveto(681.5,351.49999738)
\curveto(681.5,323.88575989)(596.43410247,301.49999738)(491.5,301.49999738)
\curveto(386.56589753,301.49999738)(301.5,323.88575989)(301.5,351.49999738)
\lineto(301.5,351.4999997)
}
}
{
\pscustom[linewidth=2,linecolor=black,linestyle=dashed,dash=3 3]
{
\newpath
\moveto(560,351.49999738)
\curveto(560,334.93145489)(528.65993249,321.49999738)(490,321.49999738)
\curveto(451.34006751,321.49999738)(420,334.93145489)(420,351.49999738)
\curveto(420,351.62357024)(420.00178153,351.74714192)(420.00534449,351.87070534)
}
}
{
\put(350,150){$\rho'$}
\put(80,240){$R$}
\put(930,330){$x_{i-1}$}
\put(920,410){$\rho$}
\put(610,330){$A$}
\put(500,330){$x_i$}
\put(360,330){$y_0$}
\put(500,540){$p$}
\put(570,540){$D_4$}
\put(690,600){$D_3$}
}
\end{pspicture}
\caption{\label{pic3}}
\end{figure}
Repeating this for any other points of $\sigma_y\cap R$, we may arrange that $\sigma_y$ first meets $R$ when it begins to coincide with $\rho$.
The same is then true of $\rho'$.
\end{proof}

\begin{remark}
A shorter, but less instructive, proof can be given using Lemma \ref{pathdisjointlemma} and Remark \ref{actionremark}.
\end{remark}

\begin{remark}
Now that we know when two arcs give the same surface, to describe a surface we only need to find some suitable arc.
\end{remark}
\subsection{Adjacency of surfaces}

We now have a good idea of when two taut Seifert surfaces in $M_p$ are distinct. When are they adjacent in the Kakimizu complex?
Recall that we have already seen that, for two surfaces $R$ and $R'$, if $[R]$ and $[R']$ are adjacent in $\ms_p(L)$ then $\inc_*(R)$ and $\inc_*(R')$ are either adjacent or the same in $\ms(L)$.

\begin{lemma}\label{homeolemma}
Let $R$, $R'$ be taut Seifert surfaces for $L$, and let $\rho$ be a simple closed curve based at $p$.
Then $[\phi_{\rho}(R)]$ and $[\phi_{\rho}(R')]$ are adjacent in $\ms_p(L)$ if and only if $[R]$ and $[R']$ are.
\end{lemma}

\begin{lemma}\label{meetoncelemma1}
Let $R$ be a taut Seifert surface for $L$, 
and let $\rho$ be a simple closed curve based at $p$ that meets $R$ once. Then $[R]$ and $[\phi_{\rho}(R)]$ are adjacent in $\ms_p(L)$.
\end{lemma}
\begin{proof}
Since $\rho$ cannot be made disjoint from $R$, the two vertices are distinct.
Take a copy $R'$ of $R$. If $\rho$ crosses $R$ in the positive direction, push $R'$ off $R$ in the positive direction. If instead $\rho$ crosses $R$ in the negative direction, push $R'$ off $R$ in the negative direction. Now $R$ and $\phi_{\rho}(R')$ are tight.
\end{proof}

\begin{lemma}\label{meetoncelemma2}
Suppose $[R]$ and $[R']$ are adjacent in $\ms_p(L)$, and that $\inc_*(R)=\inc_*(R')$. 
Then $[R']=[\phi_{\rho}(R)]$ for a simple closed loop $\rho$ based at $p$ that meets $R$ once.
\end{lemma}
\begin{proof}
Isotope $R$ and $R'$ to be tight, making any components coincide where possible. 
Since $[R]$ and $[R']$ are distinct, there exists at least one component $R_a$ of $R$ that does not coincide with the corresponding component $R'_b$ of $R'$. 
As $R_a$ and $R'_b$ are isotopic in $M$, by Theorem \ref{productregionthm} there is a product region $T$ between them, which must contain $p$.
Thus there is exactly one such pair of components.

Let $\rho$ be a simple closed loop that begins at $p$, runs through $T$ to near $\partial R_a$, runs once around $L$ just above a meridian curve on $\partial\!\nhd(L)$ in the direction that takes it through $R'_b$ before $R_a$, and then returns to $p$ through $T$.
Then $\phi_{\rho}(R)$ is isotopic to $R'$ in $M_p$.
\end{proof}

\begin{corollary}\label{adjacencycor1}
Let $R$ be a taut Seifert surface for $L$, and $\rho$, $\rho'$ two simple closed curves based at $p$. 
Then $[\phi_{\rho}(R)]$ and $[\phi_{\rho'}(R)]$ are adjacent in $\ms_p(L)$ if and only if 
$\rho'\cdot\rho^{-1}$ can be homotoped to meet $R$ once.
\end{corollary}
\begin{proof}
By Lemma \ref{homeolemma}, $[\phi_{\rho}(R)]$ and $[\phi_{\rho'}(R)]$ are adjacent if and only if $[R]$ and $[\phi_{\rho'\cdot\rho^{-1}}(R)]$ are adjacent. 

Suppose this is the case. Then, by Lemma \ref{meetoncelemma2}, there exists a simple closed curve $\rho''$ based at $p$ that meets $R$ once such that $[\phi_{\rho'\cdot\rho^{-1}}(R)]=[\phi_{\rho''}(R)]$. Lemma \ref{pathcoincidelemma} gives a simple closed curve $\sigma$ based at $p$ and disjoint from $R$ such that $[\rho'']=[\sigma\cdot\rho'\cdot\rho^{-1}]$. Because $\sigma$ is disjoint from $R$, this gives the required result.

Conversely, suppose $\rho'\cdot\rho^{-1}$ has been homotoped to meet $R$ once. By Lemma \ref{meetoncelemma1}, $[R]$ and $[\phi_{\rho'\cdot\rho^{-1}}(R)]$ are adjacent. 
\end{proof}

We now find a copy of $\ms(L)$ inside $\ms_p(L)$.

\begin{definition}[\cite{MR2869183} Section 5]\label{orderingdefn}
Given a taut Seifert surface $S$ for $L$, define a relation $<_S$ on $\V(\ms(L))$ as follows.

Let $R,R'$ be taut Seifert surfaces for $L$ such that $[R]$ and $[R']$ are adjacent in $\ms(L)$.
Isotope the surfaces so that $R,R'$ are almost transverse to and have simplified intersection with $S$, and so that $R,R'$ are almost disjoint with simplified intersection (see Definition \ref{almostdefn}).

Let $V_S$ be a lift of $M \setminus S$.
Let $V_R$ be the lift of $M \setminus R$ such that $V_R \cap V_S\neq\emptyset$ but $V_R \cap\tau(V_S) = \emptyset$. 
Finally, let $V_{R'}$ be the lift of $M \setminus R'$ such that $V_{R'} \cap V_R \neq\emptyset$ but $V_{R'} \cap \tau(V_R) = \emptyset$.
Say $[R'] <_S [R]$ if $V_{R'} \cap V_S\neq\emptyset$.
\end{definition}

It is shown in \cite{MR2869183} that $<_S$ is well defined, that any two adjacent vertices of $\ms(L)$ are comparable, and that there are no $R_1,\ldots,R_k$, for $k\geq 2$, with $R_1 <_S R_2 <_S\cdots<_S R_k <_S R_1$.

\begin{definition}
Fix a taut Seifert surface $R_p$ for $L$.
Define a relation $\leq$ on $V(\ms(L))$ by $[R]\leq[R']$ if either $[R]=[R']$ or $[R]<_{R_p}[R']$.
\end{definition}

\begin{definition}[\cite{2011arXiv1109.0965B} Definition 6.1]
Let $R,R'$ be $\partial$--almost disjoint taut Seifert surfaces for $L$. Pick a component $K$ of $L$, and consider $R,R'$ near $K$. It may be that the components that meet $K$ coincide. If not, one of the two surfaces lies `above' the other, where this is measured in the positive direction around $K$. Write $R\leq_{K} R'$ if either $R'$ lies above $R$ or the two coincide.
\end{definition}

\begin{proposition}[\cite{2011arXiv1109.0965B} Proposition 6.4]\label{representativesprop}
It is possible to choose a representative $R^*$ (considered as fixed in $M$) for each vertex $[R]$ of $\ms(L)$ such that the following conditions hold for any pair $([R_1],[R_2])$ of adjacent vertices of $\ms(L)$.
\begin{itemize}
	\item $\partial R_1^*=\partial R_2^*=\partial R_p$.
	\item There is a representative $R'_i$ of $[R_i]$ that is isotopic to $R_i^*$ by an isotopy fixing the boundary for $i=1,2$ such that $R'_1$ and $R'_2$ are $\partial$--almost disjoint and tight.
	\item $[R_1]\leq[R_2]$ if and only if $R'_1\leq_{K}R'_2$ for every component $K$ of $L$.
\end{itemize}
\end{proposition}

Let $\mathcal{R}$ be this set of representatives. 
We may position $p$ to be disjoint from every $R\in\mathcal{R}$, and such that the isotopies given can be performed in the complement of $p$. To see this, note that all the representatives coincide on $\partial M=\partial\! \nhd(L)$, and that the isotopies also fix the boundary. We can therefore find a suitable position for $p$ close to the boundary. That is, we take $p$ to lie in a component of $M\setminus(\bigcup_{R\in\mathcal{R}} R)$ that meets $\partial M$ away from $\partial R_p^*$.
We now consider these surfaces to all lie in $M_p$.
They all represent distinct vertices of $\ms_p(L)$, since their images under $\inc_*$ are distinct.

Suppose $\inc_*(R)$ and $\inc_*(R')$ are adjacent in $\ms(L)$ for some $R,R'\in\mathcal{R}$. 
Then they can be isotoped in $M_p$ to be tight in $M$.
This shows that they are also adjacent in $\ms_p(L)$.
Hence the subcomplex of $\ms_p(L)$ induced by $\mathcal{R}$ is an embedded copy of $\ms(L)$.

\begin{definition}
For each taut Seifert surface $R$ in $M_p$, denote by $R^*$ the unique $R'\in\mathcal{R}$ such that $\inc_*(R)=\inc_*(R')$.
\end{definition}

Let $R$ be an incompressible Seifert surface for $L$. Choose a product neighbourhood $R\times[0,1]$ for $R$. Then the maps $\inc_0\colon\pi_1(R\times\{0\})\to\pi_1(M\setminus R)$ and $\inc_1\colon\pi_1(R\times\{1\})\to\pi_1(M\setminus R)$ induced by inclusion are injective.
Now $\pi_1(M)$ is given by the HNN extension 
$\langle\pi_1(M\setminus R),t\mid t^{-1}\cdot\inc_0([\rho])\cdot t=\inc_1([\rho])\forall[\rho]\in\pi_1(R)\rangle$ (see \cite{MR0577064} page 180).
The following result therefore shows that the inclusion $M\setminus R\hookrightarrow M$ induces an injection $\pi_1(M\setminus R,p)\hookrightarrow\pi_1(M,p)$.

\begin{theorem}[\cite{MR0577064}  Chapter IV Theorem 2.1]
Let $G$ be a group, and $H_1, H_2\leq G$. Let $\phi\colon H_1\to H_2$ be an isomorphism. 
Then the map $g\mapsto g$ embeds 
$G$ into $\langle G, t\mid t^{-1}ht=\phi(h)\forall h\in H_1\rangle$.
\end{theorem}

\begin{definition}
For $R\in\mathcal{R}$, let $X_{R}=\pi_1(M,p)\mathclose{}/\mathopen\sim_R$ where $[\sigma_1]\sim_R[\sigma_2]$ if and only if there exists $\sigma_3\in\pi(M\setminus R,p)$ such that $[\sigma_1]=[\sigma_3\cdot\sigma_2]$.
In addition, let $X=\{(R,[\rho]):R\in\mathcal{R}\textrm{, }[\rho]\in X_{R}\}$.
\end{definition}

\begin{corollary}
There is a bijection $\phi\colon X\to\V(\ms(L))$, given by $\phi(R,[\rho])= [\phi_{\rho}(R)]$.
\end{corollary}
\begin{proof}
Let $R\in\mathcal{R}$, and let $\rho$, $\rho'$ be simple closed curves based at $p$ such that $[\rho]\sim_R[\rho']$. Then there exists a simple closed curve $\sigma$ based at $p$ and disjoint from $R$ such that $[\rho]=[\sigma\cdot\rho']$.
By Remark \ref{disjointpathremark}, $[\phi_{\sigma}(R)]=[R]$, so $[\phi_{\rho}(R)]=[\phi_{\sigma\cdot\rho'}(R)]=[\phi_{\rho'}(R)]$.
Thus $\phi$ is well-defined.

Now let $R,R'\in\mathcal{R}$ and let $\rho$, $\rho'$ be simple closed curves based at $p$, such that $[\phi_{\rho}(R)]=[\phi_{\rho'}(R')]$.
Then $R=\phi_{\rho}(R)^*=\phi_{\rho'}(R')^*=R'$.
Given this, Lemma \ref{pathcoincidelemma} applies, showing that $[\rho]\sim_R[\rho']$.
Therefore $\phi$ is injective.

Finally, let $R'$ be a taut Seifert surface for $L$ in $M_p$. Let $R=(R')^*$. Then $\inc_*(R)=\inc_*(R')$.
By Lemma \ref{pathexistslemma} there exists a simple closed path $\rho$ based at $p$ such that $[\phi_{\rho}(R)]=[R']$. Hence $\phi$ is surjective.
\end{proof}

\begin{remark}
Let $R\in\mathcal{R}$. Then any element of $\pi_1(M\setminus R,p)$ has linking number 0 with $L$. Therefore the map $\lk\colon\pi_1(M,p)\to\mathbb{Z}$ descends to a well-defined surjection from $X_R$ to $\mathbb{Z}$.
\end{remark}

\begin{lemma}[\cite{2011arXiv1109.0965B} Lemma 4.1.9]\label{productregionsidelemma}
Let $R$, $R'$ be taut Seifert surfaces for $L$ such that $[R]$ and $[R']$ are adjacent in $\ms(L)$. Then $R$, $R'$ can be isotoped so they are disjoint and tight.
Suppose there are components $R_a$ of $R$ and $R'_b$ of $R'$ that can be made to coincide, so there is a product region between these components. The side of $R$ on which this product region lies is determined by the choice of $([R],[R'])$.
\end{lemma}

\begin{lemma}\label{isotopeincomplementlemma}
Let $R_+$ be a union of disjoint taut Seifert surfaces for $L$. Let $R'_a, R'_b$ be copies of the same component of a taut Seifert surface $R'$ for $L$. Suppose that $R'_a$ and $R'_b$ are disjoint from $R_+$.
Further assume either that $R'_a$ is not parallel to any component of $R_+$ or that both some copy $R'_A$ of $R'$ containing $R'_a$ and some copy $R'_B$ of $R'$ containing $R'_b$ are tight relative to each of the Seifert surfaces that make up $R_+$. 
Then $R'_a$ is isotopic to $R'_b$ in the complement of $R_+$.

In addition, if $R'_a$ and $R'_b$ are disjoint then there is a product region between them that is disjoint from $R_+$.
\end{lemma}
\begin{proof}
Isotope $R'_a$ in the complement of $R_+$ to minimise its intersection with $R'_b$.
Since they are isotopic in $M$, by Theorem \ref{productregionthm} there is a product region $T$ between them.
If $T$ meets a curve of $R'_a\cap R'_b$ then $T\cap R_+=\emptyset$. This means that $T$ can be used to find an isotopy of $R'_a$ in the complement of $R_+$ that reduces $|R'_a\cap R'_b|$, which we have assumed cannot be done.
Hence $R'_a$ and $R'_b$ are disjoint.
If $T\cap R_+=\emptyset$ then $T$ gives an isotopy of $R'_a$ in the complement of $R_+$ that makes it coincide with $R'_b$.

Suppose instead that $T\cap R_+\neq\emptyset$. 
Let $R_c$ be a component of $R_+$ that meets $T$, and let $R_C$ be the Seifert surface in $R_+$ that contains $R_c$.
By Proposition \ref{surfaceinproductprop}, $R_c$ divides $T$ into two product regions $T_a$ and $T_b$, where $T_a$ lies between $R'_a$ and $R_c$ while $T_b$ lies between $R_c$ and $R'_b$.
Under the hypothesis that $R'_a$ is not parallel to any component of $R_C$, this cannot occur.
On the other hand, if both $R'_A$ and $R'_B$ are tight relative to $R_C$, Lemma \ref{productregionsidelemma} applies to the pair of vertices $([R_C],[R'])$, again giving a contradiction.
\end{proof}

\begin{proposition}\label{adjacenttopathsprop}
Let $R, R'\in\mathcal{R}$ such that $\inc_*(R)$ and $\inc_*(R')$ are adjacent in $\ms(L)$. Take copies $R_0$ of $R$ and $R'_0$ of $R'$, and isotope them in $M_p$ to be tight keeping the boundary fixed.
Let $\rho$, $\rho'$ be simple closed curves based at $p$.
Suppose $[\phi_{\rho}(R)]$, $[\phi_{\rho'}(R')]$ are adjacent.
Then there exists a path $\sigma_R$ disjoint from $R_0$ and a path $\sigma_{R'}$ disjoint from $R'_0$ such that $\sigma_{R'}\cdot\sigma_R$ is a simple closed loop based at $p$ and
$[\phi_{\rho'}(R')]=[\phi_{\sigma_{R'}\cdot\sigma_{R}\cdot\rho}(R')]$.
\end{proposition}
\begin{proof}
Let $R''$ be a copy of $\phi_{\rho'\cdot\rho^{-1}}(R'_0)$.
By Lemma \ref{homeolemma}, $[\phi_{\rho'\cdot\rho^{-1}}(R')]$ is adjacent to $[R]$. Isotope $R''$ so that $R_0$ and $R''$ are tight and disjoint.
We wish to relate $R''$ to $R'_0$ so we find a suitable representative for $[\rho'\cdot\rho^{-1}]$.

As $R''$ and $R'_0$ are isotopic in $M$,
by Lemma \ref{isotopeincomplementlemma} they are isotopic in $M\setminus R_0$.
The path $\sigma_R^{-1}$ of $p$ under the isotopy taking $R''$ to $R'_0$ is disjoint from $R_0$.

We would like to take $\sigma_R$ as a replacement for $\rho'\cdot\rho^{-1}$.
However, it may not be possible to arrange that $\sigma_R$ begins at $p$, as it may lie in the wrong component of $M\setminus(R_0\cup R'_0)$.
On the other hand, we can arrange that $\sigma_R$ begins in a neighbourhood of $\partial M$.

Choose a second path $\sigma_{R'}$ that is disjoint from $R'_0$ such that $\sigma_{R'}\cdot\sigma_R$ is a simple closed loop.
Then $[\phi_{\rho'\cdot\rho^{-1}}(R')]=[R'']=[\phi_{\sigma_{R'}\cdot\sigma_R}(R')]$,
so $[\phi_{\rho'}(R')]=[\phi_{\sigma_{R'}\cdot\sigma_R\cdot\rho}(R')]$.
\end{proof}

\begin{remark}\label{boundintersectionremark}
We can bound the number of times $\sigma_{R'}$ meets $R_0$. More precisely, we can ensure this number is at most the number of components of $R'_0$ that $\sigma_R$ meets, and in particular at most the number of components of $R'$.  To do this, we begin with a copy of $\sigma_R^{-1}$ and modify it to arrive at $\sigma_{R'}$, as follows.

Run along $\sigma_R^{-1}$ until the point $x$ where it first meets $R'_0$. Find the last point $y$ at which $\sigma_{R}^{-1}$ meets this component of $R'_0$. Choose paths in $R'_0$ from $x$ and $y$ to the same point $z$ on $\partial R'_0$. Depending on which side of $R'_0$ the two sections of $\sigma_{R}^{-1}$ lie, either we can join these two paths together at $z$ in the complement of $R'_0$ or we can join them by a path that runs once around a meridian of $L$, which will cross $R_0$ exactly once.
Now continue along $\sigma_R^{-1}$ from $y$ to where it next meets $R'_0$ and repeat this process. Once the end of $\sigma_R^{-1}$ is reached, the resulting path $\sigma_{R'}$ has the required form.
\end{remark}

\begin{example}
For $n\geq 1$, we can give an example of a link $L_n$ where we may construct surfaces $R$ and $R'$, and a path $\sigma_R$, such that $\sigma_{R'}$ must cross $R$ at least $n$ times. We form $L_n$ using copies of the knot $7_4$. This knot has two distinct taut Seifert surfaces. Once made disjoint, these together cut the knot complement into two genus 2 handlebodies. The knot fits in a natural way into such a handlebody, so we may nest another copy of $7_4$ within one of them. We may repeat this until the link has $n$ components. Put one of the two surfaces for each link component into $R$, and put the other into $R'$, in such a way as to ensure that $R$ and $R'$ are tight (alternatively, choose which surfaces lie in $R$, then choose the orientations of the link components to ensure $R$ and $R'$ are tight). Take $p=\sigma_R(1)$ to lie in one of the two remaining outer handlebodies, and $\sigma_R(0)$ in the other, with $\sigma_R$ running once through each component of $R'$. Then to travel from $p$ to $\sigma_R(0)$ the path $\sigma_{R'}$ must cross every component of $R$.
\end{example}

\begin{remark}
Because linking number with a knot or link is measured by algebraic intersection with a Seifert surface, and $\sigma_R$ is disjoint from $R_0$, it is usually the case that if $\sigma_R$ meets a component $R'_b$ of $R'_0$ more than once then any two consecutive such points along $\sigma_R$ must cross $R'_b$ in opposite directions. This might not be true, however, if the components of $R$ and $R'$ link different combinations of components of $L$. A simple example of a link for which this can occur is the $(4,4)$ torus link.
\end{remark}

\begin{lemma}\label{pathpathadjacentlemma}
Let $R, R'\in\mathcal{R}$ such that $\inc_*(R)$ and $\inc_*(R')$ are adjacent in $\ms(L)$. Take copies $R_0$ of $R$ and $R'_0$ of $R'$, and isotope them to be tight keeping the boundary fixed. Let $\rho$ be a simple closed curve based at $p$. 
Let $\sigma_R$ be a path disjoint from $R_0$ and $\sigma_{R'}$ a path disjoint from $R'_0$ such that $\sigma_{R'}\cdot\sigma_R$ is a simple closed curve based at $p$.
Then $[\phi_{\rho}(R)]$ and $[\phi_{\sigma_{R'}\cdot\sigma_R\cdot\rho}(R')]$ are adjacent in $\ms_p(L)$.
\end{lemma}
\begin{proof}
Let $R''$ be a copy of $R'_0$.
Consider how the position of $R''$ changes as $p$ moves along $\sigma_{R'}$ and $\sigma_R$.
Since $\sigma_{R'}$ is disjoint from $R'_0$, moving $p$ along $\sigma_{R'}$ does not change $R''$.
As $\sigma_{R}$ is disjoint from $R_0$, changing $R''$ when moving $p$ along it does not stop $R_0$ and $R''$ being tight.  
Then $\phi_{\rho}(R_0)$ and $\phi_{\rho}(R'')$ are also tight.
Thus $[\phi_{\rho}(R)]$ is adjacent to $[\phi_{\rho}(R'')]=[\phi_{\sigma_{R'}\cdot\sigma_R\cdot\rho}(R')]$.
\end{proof}
\subsection{Connected surfaces}

For the remainder of this section we will assume that every taut Seifert surface for $L$ is connected.
This is true in particular for all knots. 
In Proposition \ref{adjacenttopathsprop}, using Remark \ref{boundintersectionremark}, we find that
$M\setminus(R_0\cup R'_0)$ has two components and $\sigma_{R'}$ meets $R_0$ at most once.
In addition, the direction in which $\sigma_R$ crosses $R'_0$ alternates as measured along $\sigma_R$.

\begin{corollary}\label{adjacencycor2}
Let $R,R'\in\mathcal{R}$ and let $\rho$, $\rho'$ be simple closed curves based at $p$. Then $[\phi_{\rho}(R)]$ is adjacent to $[\phi_{\rho'}(R')]$ in $\ms_p(L)$ if and only if one of the following holds.
\begin{itemize}
\item[(1)] $R=R'$ and $\rho'\cdot\rho^{-1}$ can be homotoped to meet $R$ exactly once.
\item[(2)] $\inc_*(R)$, $\inc_*(R')$ are adjacent in $\ms(L)$, and there exists a simple closed curve $\rho''$ based at $p$ such that $[\rho]\sim_R[\rho'']$ and $[\rho']\sim_{R'}[\rho'']$.
\item[(3)] $\inc_*(R)$, $\inc_*(R')$ are adjacent in $\ms(L)$, and after isotoping copies $R_0$ of $R$ and $R'_0$ of $R'$ to be tight keeping the boundary fixed there exist simple closed curves $\rho''$ and $\sigma$ based at $p$ such that $[\rho]\sim_R[\rho'']$, $[\rho']\sim_{R'}[\sigma\cdot\rho'']$, and $\sigma$ meets $R_0$ and $R'_0$ each once in that order.
\end{itemize}
\end{corollary}
\begin{proof}
If $R=R'$ then
(1) is given by Corollary \ref{adjacencycor1}.
Suppose instead that $R\neq R'$. 

First assume that $[\phi_{\rho}(R)]$ and $[\phi_{\rho'}(R')]$ are adjacent. Then $\inc_*(R)$ and $\inc_*(R')$ are adjacent in $\ms(L)$.
Apply Proposition \ref{adjacenttopathsprop}.
If $\sigma_{R'}$ is disjoint from $R_0$, then (2) holds with $\rho''=\sigma_{R'}\cdot\sigma_R\cdot\rho$. 

Suppose that $\sigma_{R'}$ meets $R_0$ once. Then $\sigma_R$ meets $R'_0$ an odd number of times. Let $x$ be a point on $\sigma_R$ just after it first crosses $R'_0$. Then $x$ lies in the same component of $M\setminus(R_0\cup R'_0)$ as $p$.
By moving $x$ to $p$ in $M\setminus(R_0\cup R'_0)$, isotope $\sigma_R$ to $\sigma_a\cdot\sigma_b$, where $\sigma_a$ is a path from $\sigma_R(0)$ to $p$ that is disjoint from $R_0$ and meets $R'_0$ once, and $\sigma_b$ is a simple closed curve based at $p$ that is disjoint from $R_0$.
Let $\rho''=\sigma_b\cdot\rho$ and $\sigma=\sigma_{R'}\cdot\sigma_a$.
Then $[\rho]\sim_R[\rho'']$ and $[\rho']\sim_{R'}[\sigma\cdot\rho'']$, so (3) holds.

Conversely, suppose that (2) holds. By Lemma \ref{homeolemma}, $[\phi_{\rho}(R)]=[\phi_{\rho''}(R)]$ is adjacent to $[\phi_{\rho''}(R')]=[\phi_{\rho'}(R')]$.

Finally, suppose that (3) holds. Let $x$ be a point on $\sigma$ between where it meets $R$ and where it meets $R'$. Let $\sigma_R$ be the section of $\sigma$ up to $x$, and $\sigma_{R'}$ the section after $x$. Applying Lemma \ref{pathpathadjacentlemma} shows that $[\phi_{\rho}(R)]=[\phi_{\rho''}(R)]$ is adjacent to  $[\phi_{\sigma\cdot\rho''}(R')]=[\phi_{\rho'}(R')]$.
\end{proof}

\begin{remark}
If (3) holds, then $[\phi_{\sigma\cdot\rho''}(R)]$ is adjacent to $[\phi_{\rho}(R)]$ and to $[\phi_{\rho'}(R')]$, giving a 2--simplex in $\ms_p(L)$. Similarly, $[\phi_{\rho''}(R')]$ is also adjacent to $[\phi_{\rho}(R)]$ and to $[\phi_{\rho'}(R')]$.
\end{remark}

Suppose (3) holds. If $\inc_*(R)\leq \inc_*(R')$ then $\sigma$ crosses $R$ and $R'$ in the positive direction and $\lk(\rho')-\lk(\rho)=\lk(\sigma)=1$. If $\inc_*(R)\geq \inc_*(R')$ then $\sigma$ crosses $R$ and $R'$ in the negative direction and $\lk(\rho')-\lk(\rho)=\lk(\sigma)=-1$.
Similarly, if (1) holds then $|\lk(\rho)-\lk(\rho')|=1$.
On the other hand, if (2) holds then $\lk(\rho)=\lk(\rho')$.
In each case, if $\inc_*(R)<_{R_p}\inc_*(R')$ then $\lk(\rho)\leq\lk(\rho')$.

\begin{theorem}
The dimension of $\ms_p(L)$ is one higher than the dimension of $\ms(L)$.
\end{theorem}
\begin{proof}
Let $(R_0,[\rho_0]),\ldots,(R_n,[\rho_n])\in X$ such that $[\phi_{\rho_0}(R_0)],\ldots,[\phi_{\rho_n}(R_n)]$ are the vertices of an $n$--simplex in $\ms_p(L)$. 
Then $\inc_*(R_0),\ldots,\inc_*(R_n)$ are the vertices of a simplex in $\ms(L)$.
Without loss of generality, $\inc_*(R_0)\leq \inc_*(R_1)\leq \cdots\leq \inc_*(R_n)$ and $\lk(\rho_0)\leq\lk(\rho_1)\leq\cdots\leq\lk(\rho_n)$.
Since $0\leq\lk(\rho_n)-\lk(\rho_0)\leq 1$ there exists $k\leq n$ such that $\lk(\rho_i)=\lk(\rho_0)$ if $i\leq k$ and $\lk(\rho_i)=\lk(\rho_0)+1$ if $i> k$.

If $R_j=R_{j+1}$ for some $j$ then $j=k$, since $(R_j,[\rho_j])$ and $(R_{j+1},[\rho_{j+1}])$ are distinct.
Hence there is at most one such $j$. This means that the dimension of $\ms_p(L)$ is at most one higher than the dimension of $\ms(L)$.

Conversely, let $R_0,\ldots,R_n\in\mathcal{R}$ be distinct and such that $\inc_*(R_0),\ldots,\inc_*(R_n)$ are the vertices of a maximal dimensional simplex in $\ms(L)$. 
Without loss of generality, $\inc_*(R_0)\leq \inc_*(R_1)\leq \cdots\leq \inc_*(R_n)$.
Let $\rho$ be any simple closed curve based at $p$, and let $\sigma$ be the simple closed curve based at $p$ that runs once around the meridian of $L$ in the positive direction (recall that we positioned $p$ near $\partial M$).
Then, by Corollary \ref{adjacencycor2}, $[\phi_{\rho}(R_0)],\ldots,[\phi_{\rho}(R_n)],[\phi_{\sigma\cdot\rho}(R_n)]$ are all pairwise adjacent, and so are the vertices of an $(n+1)$--simplex in $\ms_p(L)$.
\end{proof}

The following result is a version of Corollary \ref{adjacencycor2} that allows us more freedom to position the surfaces of interest.
We will make use of this freedom in the proof of Proposition \ref{threedisjointprop}.
\begin{corollary}\label{adjacencycor3}
Let $R,R'\in\mathcal{R}$ and let $\rho$, $\rho'$ be simple closed curves based at $p$. 
Suppose that $[\phi_{\rho}(R)]$ and $[\phi_{\rho'}(R')]$ are adjacent in $\ms_p(L)$.
Let $R_0$, $R'_0$ be copies of $R$, $R'$ respectively that have been isotoped in $M_p$ to be tight keeping the boundary fixed. If $R=R'$ then ensure that $R_0$ and $R'_0$ coincide.
Then one of the following holds.
\begin{itemize}
\item[(1)] $R_0=R'_0$ and $\rho'\cdot\rho^{-1}$ can be homotoped to meet $R_0$ exactly once.
\item[(2)] $\inc_*(R_0)$, $\inc_*(R'_0)$ are adjacent in $\ms(L)$, and there exists a simple closed curve $\rho''$ based at $p$ such that $[\rho]\sim_R[\rho'']$ and $[\rho']\sim_{R'}[\rho'']$.
\item[(3)] $\inc_*(R_0)$, $\inc_*(R'_0)$ are adjacent in $\ms(L)$, and 
there exist simple closed curves $\rho''$ and $\sigma$ based at $p$ such that $[\rho]\sim_R[\rho'']$, $[\rho']\sim_{R'}[\sigma\cdot\rho'']$, and $\sigma$ meets $R_0$ and $R'_0$ each once in that order.
\end{itemize}
\end{corollary}

\begin{theorem}[\cite{MR0577064} Chapter IV Theorem 2.6]\label{freeproductthm}
Let $G_1,G_2$ be groups, $H_1\leq G_1$ and $H_2\leq G_2$. Let $\phi\colon H_1\to H_2$ be an isomorphism.
If $g_1\in G_1$, $g_2\in G_2$ and $g_1\cdot g_2=1$ in $\langle G_1, G_2 \mid h_1=\phi(h_1)\forall h_1\in H_1\rangle$ then $g_1\in H_1$.
\end{theorem}

\begin{proposition}\label{threedisjointprop}
Let $(R_0,[\rho_0]),(R_1,[\rho_1]),(R_2,[\rho_2])\in X$ such that $R_1$, $R_2$, $R_3$ can be isotoped in $M_p$ to be pairwise tight (keeping the boundary fixed), and such that $[\phi_{\rho_0}(R_0)]$, $[\phi_{\rho_1}(R_1)]$, $[\phi_{\rho_2}(R_2)]$ span a 2--simplex in $\ms_p(L)$.
Let $R'_0$, $R'_1$, $R'_2$ be copies of $\phi_{\rho_0}(R_0)$, $\phi_{\rho_1}(R_1)$, $\phi_{\rho_2}(R_2)$ respectively.
Then $R'_0$, $R'_1$, $R'_2$ can be isotoped in $M_p$ keeping the boundary fixed so that, after the isotopy, each pair is tight.
\end{proposition}
\begin{proof}
Assume that
$\inc_*(R_0)\leq \inc_*(R_1)\leq \inc_*(R_2)$ and $\lk(\rho_0)\leq \lk(\rho_1)\leq\lk(\rho_2)$.
Since $\lk(\rho_0)$, $\lk(\rho_1)$, $\lk(\rho_2)$ take at most 2 values, by the symmetry in Corollary \ref{adjacencycor3} we may also assume that $\lk(\rho_1)=\lk(\rho_0)$ and either $\lk(\rho_2)=\lk(\rho_0)$ or $\lk(\rho_2)=\lk(\rho_0)+1$.
We additionally assume that $R_0$, $R_1$, $R_2$ are already pairwise tight (with no further isotopy),  and that if two of these surfaces can be made to coincide then they already coincide.

By Corollary \ref{adjacencycor3}, there exists a simple closed curve $\rho_{01}$ based at $p$ such that $[\rho_0]\sim_{R_0^*}[\rho_{01}]$ and $[\rho_1]\sim_{R_1^*}[\rho_{01}]$.
Applying the homeomorphism $\phi_{\rho_{01}}$ shows that we can further reduce to the case where $[\phi_{\rho_0}(R_0)]=[R_0]$ and $[\phi_{\rho_1}(R_1)]=[R_1]$.

Let $\widetilde{R}_0$, $\widetilde{R}_1$, $\widetilde{R}_2$ be lifts of $R_0$, $R_1$, $R_2$ respectively to $\widetilde{M}$ such that $\partial\widetilde{R}_0=\partial\widetilde{R}_1=\partial\widetilde{R}_2$.
Each of $\widetilde{R}_0$, $\widetilde{R}_1$, $\widetilde{R}_2$ is separating in $\widetilde{M}$.
Let $U$ be the section of $\widetilde{M}$ above $\widetilde{R}_1$, and $V$ the section below $\tau(\widetilde{R}_0)$. Also let $W=U\cap V$, and let $\widetilde{p}$ be the lift of $p$ in $W$. Note that $\widetilde{R}_2$ separates $W$ and $\widetilde{p}$ lies above $\widetilde{R}_2$.

As $\widetilde{R}_1$ and $\tau(\widetilde{R}_0)$ are incompressible, the maps $\inc_U\colon\pi_1(W)\to\pi_1(U)$ and $\inc_V\colon\pi_1(W)\to\pi_1(V)$ induced by inclusion are injective. Thus $\pi_1(\widetilde{M})$ is given by the amalgamated free product 
$\langle \pi_1(U),\pi_1(V)\mid\inc_U([\rho])=\inc_V([\rho])\forall[\rho]\in\pi_1(W)\rangle$ (see \cite{MR0577064} page 180).

First suppose that $\lk(\rho_2)=\lk(\rho_0)$. Then, by Corollary \ref{adjacencycor3}, there exists a simple closed curve $\rho_{02}$ based at $p$ such that $[\rho_0]\sim_{R_0^*}[\rho_{02}]$ and $[\rho_2]\sim_{R_2^*}[\rho_{02}]$. This means that (after a homotopy) $\rho_{02}$ is disjoint from $R_0$. Similarly there exists a simple closed curve $\rho_{12}$ based at $p$ and disjoint from $R_1$ such that $[\phi_{\rho_2}(R_2)]=[\phi_{\rho_{12}}(R_2)]$.
From Lemma \ref{pathcoincidelemma} we see that $[\rho_{02}]=[\sigma\cdot\rho_{12}]$ for some simple closed curve $\sigma$ based at $p$ and disjoint from $R_2$.

As $\rho_{02}$, $\rho_{12}$, $\sigma$ are each disjoint from some Seifert surface for $L$, they all lift to closed curves in $\widetilde{M}$.
Let $\widetilde{\rho}_{02}$, $\widetilde{\rho}_{12}$, $\widetilde{\sigma}$ be the lifts of $\rho_{02}$, $\rho_{12}$, $\sigma$ respectively that are based at $\widetilde{p}$.
Notice that $[\widetilde{\rho}_{02}]\in\pi_1(V)$ and $[\widetilde{\sigma}\cdot\widetilde{\rho}_{12}]\in\pi_1(U)$.
Because $[\widetilde{\rho}_{02}]^{-1}\cdot[\widetilde{\sigma}\cdot\widetilde{\rho}_{12}]=1$ in $\pi_1(\widetilde{M})$, Theorem \ref{freeproductthm} gives that $[\widetilde{\rho}_{02}]\in\pi_1(W)$.
Homotope $\rho_{02}$ so that $\widetilde{\rho}_{02}$ lies in $W$. 
Then $\rho_{02}$ is disjoint from $R_0\cup R_1$.
Hence $\phi_{\rho_{02}}(R_2)$ is tight with respect to $R_0$ and with respect to $R_1$.

Now suppose instead that $\lk(\rho_2)=\lk(\rho_0)+1$. 
In this case it may be that $[R_2]=[R_1]$. If this is so then $R_1$ and $R_2$ coincide.
In any case, by Corollary \ref{adjacencycor3}, there exist simple closed curves $\rho_{02}$, $\sigma_{02}$ based at $p$ such that $\rho_{02}$ is disjoint from $R_0$, $\sigma_{02}$ meets $R_0$ and $R_2$ once each in that order, and $[\phi_{\rho_2}(R_2)]=[\phi_{\rho_{02}}(R_2)]$. Note that $\sigma_{02}$ crosses $R_0$ and $R_2$ in the positive direction.
Similarly there exist analogous curves $\rho_{12}$ and $\sigma_{12}$. 
Then, by Lemma \ref{pathcoincidelemma}, $[\sigma_{02}\cdot\rho_{02}]=[\sigma\cdot\sigma_{12}\cdot\rho_{12}]$ for some simple closed curve $\sigma$ based at $p$ and disjoint from $R_2$.

Let  $\rho=\sigma_{12}^{-1}\cdot\sigma^{-1}\cdot\sigma_{02}\cdot\rho_{02}$. Then $\lk(\rho_{12})=\lk(\rho)=0$, so both curves lift to closed curves in $\widetilde{M}$. 
Let $\widetilde{\rho}_{12}$ and $\widetilde{\rho}$ be the lifts of $\rho_{12}$ and $\rho$ respectively that start at $\widetilde{p}$.
Then $[\widetilde{\rho}_{12}]\in\pi_1(U)$ and $[\widetilde{\rho}]\in\pi_1(V)$.
As $[\widetilde{\rho}_{12}]^{-1}\cdot[\widetilde{\rho}]=1$ in $\pi_1(\widetilde{M})$, Theorem \ref{freeproductthm} gives that $[\widetilde{\rho}_{12}]\in\pi_1(W)$.
Homotope $\rho_{12}$ so that $\widetilde{\rho}_{12}$ lies in $W$.
Then $\rho_{12}$ is disjoint from $R_0\cup R_1$.
Hence $\phi_{\sigma_{12}\cdot\rho_{12}}(R_2)$ is tight with respect to $R_0$ and with respect to $R_1$.
\end{proof}

\begin{remark}
The proof of Proposition \ref{threedisjointprop} actually gives a slightly stronger result. It shows that after $R'_0$ and $R'_1$ have been isotoped to be tight we can find an isotopy of $R'_2$ to make it tight with respect to both $R'_0$ and $R'_1$.
\end{remark}

\begin{theorem}\label{realisesimplexthm}
Let $(R_0,[\rho_0]),\ldots,(R_n,[\rho_n])\in X$ be the vertices of an $n$--simplex in $\ms_p(L)$.
For $0\leq i\leq n$ let $R'_i$ be a copy of $\phi_{\rho_i}(R_i)$.
Then $R'_0,\ldots,R'_n$ can be isotoped in $M_p$ to be pairwise disjoint and tight.
\end{theorem}
\begin{proof}
We will actually isotope the surfaces so that each pair is $\partial$--almost disjoint and tight. Once this has been achieved, isotoping the boundaries of the surfaces in a neighbourhood of $\partial M$ will make the surfaces disjoint.
Assume that $\inc_*(R_0)\leq\inc_*(R_1)\leq\cdots\leq\inc_*(R_n)$ and $\lk(\rho_0)\leq\lk(\rho_1)\leq\cdots\leq\lk(\rho_n)$.

We would like it to be the case that $R_0,\ldots,R_n$ can be isotoped in $M_p$ keeping the boundary fixed to be pairwise tight. It seems difficult to prove this directly. Instead, consider the construction of the set $\mathcal{R}$. The elements of $\mathcal{R}$ were chosen using Proposition \ref{representativesprop}, the conclusions of which continue to hold if we change any $R\in\mathcal{R}$ by an isotopy in $M$ fixing the boundary.

Since $\inc_*(R_0),\ldots,\inc_*(R_n)$ span a simplex in $\ms(L)$, by Lemma \ref{alltightbdylemma} the surfaces $R_0,\ldots,R_n$ can be isotoped in $M$ keeping the boundary fixed so that after the isotopy they are pairwise tight. 
We may therefore assume that this isotopy had been performed before the choice of $R_0^*,\ldots,R_n^*$ was fixed, giving our desired result.
Note that in making this assumption we may have altered the labelling of the vertices of $\ms_p(L)$, but the complex itself is unaffected. Thus the hypotheses of Theorem \ref{realisesimplexthm} still hold, although the curves $\rho_0,\ldots,\rho_n$ may have changed.

Fix $R'_0$. 
Using Proposition \ref{threedisjointprop}, isotope $R'_1$ and $R'_2$ in $M_p$, keeping the boundary fixed, so that $R'_0$, $R'_1$ and $R'_2$ are pairwise tight.
Now isotope $R'_3$ to be tight with respect to $R'_0$ and $R'_2$.
Since, by Proposition \ref{surfacesinorderprop}, $R'_3$ and $R'_1$ lie in different components of $M_p\setminus (R'_0\cup R'_2)$, they are also tight.
Repeating this for $R'_4,\ldots,R'_n$ in order gives the required result.
\end{proof}


\section{Seifert surfaces for a split link}\label{splitlinksection}
\subsection{The general case}

Assume $L$ is a split link.
Express $L$ as a distant union of $L_1,\ldots,L_N$, where each $L_i$ is a non-split link. Then $N\geq 2$.
We will call $L_1,\ldots,L_N$ the \textit{non-split components} of $L$.
Let $R$ be an incompressible Seifert surface for $L$ (in $\Sphere$).

\begin{lemma}\label{bounddisclemma}
Let $S\subset M$ be a sphere.  
Isotope $S$ to be in general position relative to $R$, so that $S\cap R$ consists of finitely many disjoint simple closed curves.
Then each curve of $S\cap R$ bounds a disc in $R$.
\end{lemma}
\begin{proof}
We proceed by induction on $|S\cap R|$.
If $|S\cap R|\geq 1$, let $\rho$ be a curve of $S\cap R$ that is innermost in $S$. 
Then $\rho$ bounds a disc $D_1$ in $S$ with interior disjoint from $R$. As $R$ is incompressible, $\rho$ bounds a disc $D_2$ in $R$. 

Let $\rho'$ be a curve of $D_2\cap S$ that is innermost in $D_2$. Then $\rho'$ bounds a subdisc $D_3$ of $D_2$ that is disjoint from $S$ on its interior. Compress $S$ along $D_3$. This gives two spheres $S_0,S_1$, each of which have at most $|S\cap R|-1$ curves of intersection with $R$.
Note that $S\cap R$=$(S_0\cap R)\cup (S_1\cap R)\cup \rho'$.
Inductively, for $i=1,2$, every curve of $S_i\cap R$ bounds a disc in $R$. 
\end{proof}

\begin{lemma}\label{separatelinkcomponentslemma}
There is a sphere $S$ disjoint from $R$ that separates $L_2\cup \cdots\cup L_N$ from $L_1$.
\end{lemma}
\begin{proof}
Among all spheres that separate $L_2\cup \cdots\cup L_N$ from $L_1$, choose a sphere $S_0$ that has
minimal intersection with $R$. 
Let $V\subset\Sphere$ be the 3--ball bounded by $S_0$ containing $L_1$. We will view $V$ as being inside $S_0$.
Assume that $S_0\cap R\neq\emptyset$.
Let $\rho$ be a curve of $S_0\cap R$ that is innermost in $S_0$. Then $\rho$ bounds a disc $D_1$ in $S_0$ and a disc $D_2$ in $R$ as in Lemma \ref{bounddisclemma}.
Again choose $\rho'$ and $D_3$. The curve $\rho'$ bounds a disc $D_4$ in $S_0$ that does not contain $D_1$. The sphere $D_3\cup D_4$ bounds a 3--ball $V'\subset\Sphere$ with interior disjoint from $S_0$. Since $S_0\cap R$ is minimal, $V'$ must contain at least one non-split component of $L$. 
If this includes $L_1$ then $V'\subset V$ and $D_3\cup D_4$ is a sphere separating $L_1$ from $L_2\cup\cdots\cup L_N$ that has fewer curves of intersection with $R$ than $S_0$ does. This is not the case, so $L_1$ does not lie in $V'$. In particular, $V'$ lies outside of $S_0$.

Consider the component $R_0$ of $R\setminus S_0$ that meets $D_3$ along $\rho'$. As $R$ has no closed components, $R_0$ has at least two boundary components. Either $R_0$ is properly embedded in $V$ or one boundary component of $R_0$ runs along $L_1$. 
In the former case, choose a path $\sigma'$ in $R_0$ from the boundary component $\rho'$ to another boundary component. 
Let $\sigma$ be a copy of $\sigma'$ pushed off $R_0$ so that it is disjoint from $R$, remains properly embedded in $V$, and has one endpoint in the interior of $D_4$. 
In the latter case, let $\sigma'$ be a path in $R_0$ from $\rho'$ to $L_1$. Take $\sigma$ to be the path that runs immediately above $\sigma'$, around $L_1$ and then back immediately below $\sigma'$. In either case, $\sigma$ is disjoint from $R$. It also has both endpoints on the inside of $S_0$ with exactly one of them in the interior of $D_4$.
Take a neighbourhood of $\sigma$ in $V$, and let $A$ be the annulus properly embedded in $V$ that forms the boundary of this neighbourhood.
Let $D_5$ be the disc that is the union of $A$, $D_3$, and the annulus in $D_4$ between the two.

Compress $S_0$ along $D_3$, and discard the component that bounds $V'$. Now compress along $D_5$ and discard the component that is further away from $L_1$. The resulting sphere separates $L_1$ from $L_2\cup\cdots\cup L_N$ and has fewer curves of intersection with $R$ than $S_0$ does. This contradicts our choice of $S_0$. Hence $S_0$ is disjoint from $R$.
\end{proof}

From this we see that $R$ may be written as $R_1\cup\cdots\cup R_N$ where $R_i$ is an incompressible Seifert surface for $L_i$. 
For each $i$, choose a sphere $S_i$ separating $R_i$ from the other $R_j$, such that these spheres are disjoint.
We can understand the possibilities for $R_i$ within $S_i$ using the results of Section \ref{complementofpointsection} and the following lemma.

\begin{lemma}
Let $R_a$ and $R_b$ be incompressible Seifert surfaces for $L_1$ that are disjoint from $S_1$.
If $R_a$ is isotopic to $R_b$ in $M$ then $R_a$ is isotopic to $R_b$ in the complement of $S_1$.
\end{lemma}
\begin{proof}
Let $M'=\Sphere\setminus\nhd(L_1)$. Let $V$ be a closed ball in $M'$ such that $V\subset\nhd(L_2)\subseteq M'\setminus M$.
Then $R_a$ is isotopic to $R_b$ in the complement of $V$.
By expanding $V$ until $\partial V=S_1$, we can modify the isotopy of $R_a$ so that $R_a$ stays disjoint from $S_1$ throughout.
\end{proof}

Let $R'=R'_1\cup\cdots\cup R'_N$ be another incompressible Seifert surface for $L$, isotoped to have minimal intersection with $S=\bigcup S_i$.

\begin{lemma}
Let $V_1$ be the 3--ball inside $S_1$.
Let $T$ be a component of $R'\cap V_1$ that does not meet $L_1$.
Then $T$ does not separate $L_1$ in $V_1$. 
\end{lemma}
\begin{proof}
The surface $T$ is properly embedded and separating in $V_1$.
By Lemma \ref{bounddisclemma}, each boundary component of $T$ bounds a disc in $R'$.
As $R'$ has no closed components, not all these discs have interiors disjoint from $T$.
Thus $T$ is contained in a disc in $R'$.

Suppose $T$ separates $L_1$ in $V_1$.
Take a component of $T\cap S_1$ that is innermost in $S_1$.
Capping off $T$ with the disc in $S_1$ and then isotoping it slightly inside $V_1$ does not change the fact that $T$ separates $L_1$.
Repeating this until there are no remaining components of $T\cap S_1$ gives a sphere in $V_1$ that separates $L_1$.
This contradicts that $L_1$ is non-split.
\end{proof}

Let $\sigma\subset M$ be a directed framed simple path properly embedded in the outside of $S_1$. 
Shrinking $S_1$ to a point turns $\sigma$ into a simple closed curve, from which we can define the map $\phi_{\sigma}$ as in Definition \ref{homeodefn}.
Recall that $\phi_{\sigma}$ is given by pushing one point around $\sigma$. 
Define an analogous map $\Phi_{\sigma}$ that is given by pushing $S_1$, with $L_1$ inside, around $\sigma$. The framing on $\sigma$ controls the rotation of $L_1$ as $S_1$ is moved along $\sigma$, and so must be chosen so that $L$ is returned to its original position.

\begin{lemma}
Changing $\sigma$ within its homotopy class relative to its endpoints does not change the isotopy class of $\Phi_{\sigma}(R)$ or of $\Phi_{\sigma}(R')$.
\end{lemma}
\begin{proof}
Choose a product neighbourhood $\nhd(S_1)$ of $S_1$. Let $V_a$ be the 3--ball inside $S_1$, containing $L_1$. Let $V_c$ be the component of $M\setminus \nhd(S_1)$ outside $S_1$, and let $V_b$ be the piece of $M$ between $V_a$ and $V_c$.
Note that $R\subset V_a\cup V_c$, while $R'\cap V_b$ is a finite number of disjoint annuli, each of which has a boundary component on each component of $\partial V_b$.

First consider $\Phi_{\sigma}(R)$. Since $R_1$ is contained in $V_a$, it is left unchanged by $\Phi_{\sigma}$. On $R_2\cup\cdots\cup R_N$, the effect of $\Phi_{\sigma}$ is the same as that of $\phi_{\sigma'}$, 
where $\sigma'$ is the simple closed curve that comes from $\sigma$ if $V_a$ is collapsed to a point.
Hence, as in Lemma \ref{homotopelooplemma}, changing $\sigma$ within its homotopy class does not change the isotopy class of $\Phi_{\sigma}(R)$.

Now consider $\Phi_{\sigma}(R')$. 
Note that $\sigma$ intersects $V_b$ only in a single arc at each of its endpoints.
By an isotopy of $R'$ we can ensure that $\sigma$ is disjoint from $R'$ in $V_b$. Then $\Phi_{\sigma}$ acts as the identity on $R'\cap\partial V_c$.
Again, within $V_a$ the surface is not changed by $\Phi_{\sigma}$, and in $V_c$ the effect of $\Phi_{\sigma}$ is the same as that of $\phi_{\sigma'}$. 

It remains to study the effect of $\Phi_{\sigma}$ on the annuli of $R'\cap V_b$.
Each annulus in $R'\cap V_b$ is pulled once around a neighbourhood of $\sigma$ in a pattern that is similar to a Dehn twist.
Figure \ref{pic4} shows this schematically.
\begin{figure}[htbp]
\centering
\input{pic4}
\caption{\label{pic4}}
\end{figure}
It is clear that an isotopy of $\sigma$ only changes $\Phi_{\sigma}(R'\cap V_b)$ by an isotopy.
The arc in bold in Figure \ref{pic4} denotes a disc properly embedded in $V_b$ that is disjoint from $\Phi_{\sigma}(R')$.
The presence of this disc shows that, as in the proof of Lemma \ref{homotopelooplemma}, we can also homotope $\sigma$ without changing the isotopy class of $\Phi_{\sigma}(R')$.

It remains only to check that changing the framing of $\sigma$ leaves $[\Phi_{\sigma}(R')]$ unchanged.
If we pick two different framings on $\sigma$, the resulting mapping classes of $M$ will differ by a solid rotation of $V_a$.
Since both framings must return $L_1$ to its original position, it is sufficient to show that rotating $V_a$ through an angle of $2\pi$ about some axis does not change the mapping class $\Phi_{\sigma}$.

Express $V_b$ as $S_1\times[0,1]$.
Let $\psi\colon S_1\times[0,1]\to S_1$ be the isotopy of $S_1$ that rotates it by $2\pi$ about the given axis,
and define $\Psi\colon V_b\to V_b$ by $\Psi(x,t)=(\psi(x,t),t)$.
The proof will be complete if we can show that $\Psi$ is isotopic to the identity on $V_b$. 

Let $x_0$ be one of the points of $S_1$ that are fixed by the rotation. Then $\Psi$ acts as the identity on $\{x_0\}\times[0,1]$, as well as on $S_1\times\{0,1\}$. 
In addition, $(S_1\times(0,1))\setminus(\{x_0\}\times(0,1))$ is an open 3--ball.
Hence, by the Alexander trick, $\Psi$ is isotopic to the identity on $S_1\times[0,1]$.
\end{proof}

\begin{proposition}\label{movehalfprop}
Suppose $R'\cap S\neq \emptyset$.
Then there is a set $J\subset\{1,\ldots,N\}$ with $|J|\leq \frac{1}{2}(N-1)$ such that the following holds.

Let $K=\{1,\ldots,N\}\setminus J$. 
There are disjoint oriented simple arcs $\sigma_j$ in $\Sphere\setminus\bigcup_{k\in K}\nhd(L_k)$, for $j\in J$, such that $\sigma_j$ is properly embedded in the outside of $\bigcup_{i\in J}S_i$ with its endpoints on $S_j$.
Let $R'_{\Phi}$ be the image of $R'$ under $\prod_{j\in J}\Phi_{\sigma_j}$.
The curves $\sigma_j$ can be chosen so that, after an isotopy, $|R'_{\Phi}\cap S|<|R'\cap S|$.
\end{proposition}
\begin{remark}
Here $\prod_{j\in J}\Phi_{\sigma_j}$ means the composition of these functions.
\end{remark}
\begin{proof}
Choose a curve $\rho_0$ of $R'\cap S$, and let $D_0$ be the disc in $R'$ bounded by $\rho_0$.
Let $\rho_1$ be a curve of $D_0\cap S$ that is innermost in $D_0$, and let $D_1$ be the subdisc of $D_0$ bounded by $\rho_1$.

Without loss of generality, $\rho_1$ lies in $S_1$.
Note that $D_1$ is disjoint from $S_i$ for $i>1$.
The curve $\rho_1$ divides $S_1$ into two discs. Each of these together with $D_1$ forms a sphere disjoint from $L$.
By definition, each non-split component of $L$ lies on exactly one side of each such sphere.
Since $D_1$ cannot be isotoped across either of the 3--balls to reduce $|R'\cap S|$, both balls must contain at least one non-split component of $L$.
Only $L_1$ lies inside $S_1$, so $D_1$ must lie outside $S_1$.
One of the two balls contains at most half of the other non-split components of $L$. 
Let $V_1$ be this ball, and let $D'_1$ be the disc of $S_1$ bounded by $\rho_1$ that is contained in $\partial V_1$. 
Let $J=\{j:L_j\subset V_1\}$.

Let $T$ be the component of $R'\setminus S$ that meets $D_1$ along $\rho_1$.
Either $T$ has another boundary component on $S_1$ or it meets $L_1$.
As in the proof of Lemma \ref{separatelinkcomponentslemma}, there is a path $\sigma_x$ properly embedded inside $S_1$ that runs parallel to $T$ from a point $x$ near the boundary of $D'_1$ to a point $y$ on $S_1$ outside $D'_1$.

Let $S'$ be a copy of the sphere $\partial V_1$ that lies slightly inside $V_1$, so that the portion of $R'$ between $\partial V_1$ and $S'$ is $m$ annuli, where $m$ is the number of curves of $R'\cap S_1$ in the interior of $D'_1$.
We can take the start of $\sigma_x$ to lie on $S'$.
Let $\sigma_y$ be a path from $y$ to $S'$ that is disjoint from $S$.
Let $\sigma$ be the path $\sigma_x\cup\sigma_y$.
Take $\Phi$ to be a homeomorphism of $\Sphere\setminus\nhd(L)$ that results from treating $S'$ and everything inside it as a point and moving it once along the path $\sigma$. 
We require $R'$ to be moved by the isotopy, but leave each sphere $S_i$ as fixed.
That is, as we do the isotopy, we allow each component of $L$ to pass through $S$ as needed, but we do not allow it to pass through $R'$, and instead move $R'$ as needed.

Consider the image $R'_{\Phi}$ of $R'$ under $\Phi$.
This contains $2m$ new curves of intersection with $S$, as each annulus that was between $\partial V_1$ and $S'$ now runs through $D'_1$ near $x$, follows $\sigma_x$, and passes through $S_1$ again near $y$.
Apart from this, no new curves of intersection have been created.

Note that the image of $D_1$ now cobounds a ball with the subdisc $D'_1$ of $S_1$.
This can be used to isotope $D_1$ across $S_1$, removing the intersection curve $\rho_1$.
This isotopy also removes the $2m$ intersection curves in the interior of $D'_1$.
Hence we now have that $|R'_{\Phi}\cap S|=|R'\cap S|-1$. 

Take $\sigma_j$ to be the path followed by $S_j$ under the isotopy used to define $\Phi$.
We can make the $\sigma_j$ disjoint by a small perturbation, without changing the effect of $\Phi$ on $R'$.
\end{proof}

\begin{remark}
In fact our proof shows that there is a fixed $m\notin J$ such that we can take each $\sigma_j$ to be a subinterval of a simple closed curve in $\Sphere\setminus\nhd(L_m)$ that is disjoint from $S_k$ for $k\notin J\cup\{m\}$.

We had to make a choice as to which 3--ball to take as $V_1$, and we chose that containing fewer non-split components. We could have chosen on a different basis. In particular, we could instead have arranged that we did not isotope $L_1$.
\end{remark}

\begin{corollary}
For $2\leq i\leq n$ there is a simple arc $\sigma_i$ properly embedded in the outside of $S_i$ such that $\prod_{i=2}^n\Phi_{\sigma_i}(R')$ can be isotoped to be disjoint from $S$.
\end{corollary}

\begin{example}
\begin{figure}[htbp]
\centering
\psset{xunit=.21pt,yunit=.21pt,runit=.21pt}
\begin{pspicture}(1644,404.00219727)
{
\pscustom[linestyle=none,fillstyle=solid,fillcolor=lightgray]
{
\newpath
\moveto(362.00002,202.01281203)
\curveto(362.00002,135.73864205)(308.27418998,82.01281203)(242.00002,82.01281203)
\curveto(175.72585002,82.01281203)(122.00002,135.73864205)(122.00002,202.01281203)
\curveto(122.00002,268.28698201)(175.72585002,322.01281203)(242.00002,322.01281203)
\curveto(308.27418998,322.01281203)(362.00002,268.28698201)(362.00002,202.01281203)
\closepath
\moveto(922.00002,202.01281203)
\curveto(922.00002,135.73864205)(868.27418998,82.01281203)(802.00002,82.01281203)
\curveto(735.72585002,82.01281203)(682.00002,135.73864205)(682.00002,202.01281203)
\curveto(682.00002,268.28698201)(735.72585002,322.01281203)(802.00002,322.01281203)
\curveto(868.27418998,322.01281203)(922.00002,268.28698201)(922.00002,202.01281203)
\closepath
\moveto(1482.00002,202.01281203)
\curveto(1482.00002,135.73864205)(1428.27418998,82.01281203)(1362.00002,82.01281203)
\curveto(1295.72585002,82.01281203)(1242.00002,135.73864205)(1242.00002,202.01281203)
\curveto(1242.00002,268.28698201)(1295.72585002,322.01281203)(1362.00002,322.01281203)
\curveto(1428.27418998,322.01281203)(1482.00002,268.28698201)(1482.00002,202.01281203)
\closepath
}
}
{
\pscustom[linewidth=4,linecolor=gray,linestyle=dashed,dash=4 4]
{
\newpath
\moveto(802.00002,292.01281727)
\lineto(242.00002,382.01281727)
\lineto(102.00002,382.01281727)
\lineto(102.00002,22.01281727)
\lineto(242.00002,22.01281727)
\lineto(802.00002,112.01281727)
}
}
{
\pscustom[linewidth=4,linecolor=gray,linestyle=dashed,dash=12 12]
{
\newpath
\moveto(242.00002,292.01281727)
\lineto(800,239.99999727)
\lineto(1362.00002,342.01281727)
\lineto(1502.00002,342.01281727)
\lineto(1502.00002,62.01281727)
\lineto(1362.00002,62.01281727)
\lineto(800,159.99999727)
\lineto(242.00002,112.01281727)
}
}
{
\pscustom[linewidth=4,linecolor=gray]
{
\newpath
\moveto(1362.00002,292.01281727)
\lineto(802.00002,222.01281727)
\lineto(212.00002,242.01281727)
\lineto(212.00002,342.01281727)
\lineto(802.00002,262.01281727)
\lineto(1362.00002,382.01281727)
\lineto(1542.00002,382.01281727)
\lineto(1542.00002,22.01281727)
\lineto(1362.00002,22.01281727)
\lineto(802.00002,142.01281727)
\lineto(212.00002,62.01281727)
\lineto(212.00002,162.01281727)
\lineto(802.00002,182.01281727)
\lineto(1362.00002,112.01281727)
}
}
{
\pscustom[linewidth=2,linecolor=black]
{
\newpath
\moveto(42.00002,202.01281727)
\lineto(1602.00002,202.01281727)
}
}
{
\pscustom[linewidth=4,linecolor=black,fillstyle=solid,fillcolor=black]
{
\newpath
\moveto(252.00002,292.01281203)
\curveto(252.00002,286.48996453)(247.5228675,282.01281203)(242.00002,282.01281203)
\curveto(236.4771725,282.01281203)(232.00002,286.48996453)(232.00002,292.01281203)
\curveto(232.00002,297.53565953)(236.4771725,302.01281203)(242.00002,302.01281203)
\curveto(247.5228675,302.01281203)(252.00002,297.53565953)(252.00002,292.01281203)
\closepath
\moveto(252.00002,112.01281203)
\curveto(252.00002,106.48996453)(247.5228675,102.01281203)(242.00002,102.01281203)
\curveto(236.4771725,102.01281203)(232.00002,106.48996453)(232.00002,112.01281203)
\curveto(232.00002,117.53565953)(236.4771725,122.01281203)(242.00002,122.01281203)
\curveto(247.5228675,122.01281203)(252.00002,117.53565953)(252.00002,112.01281203)
\closepath
\moveto(812.00002,292.01281203)
\curveto(812.00002,286.48996453)(807.5228675,282.01281203)(802.00002,282.01281203)
\curveto(796.4771725,282.01281203)(792.00002,286.48996453)(792.00002,292.01281203)
\curveto(792.00002,297.53565953)(796.4771725,302.01281203)(802.00002,302.01281203)
\curveto(807.5228675,302.01281203)(812.00002,297.53565953)(812.00002,292.01281203)
\closepath
\moveto(812.00002,112.01281203)
\curveto(812.00002,106.48996453)(807.5228675,102.01281203)(802.00002,102.01281203)
\curveto(796.4771725,102.01281203)(792.00002,106.48996453)(792.00002,112.01281203)
\curveto(792.00002,117.53565953)(796.4771725,122.01281203)(802.00002,122.01281203)
\curveto(807.5228675,122.01281203)(812.00002,117.53565953)(812.00002,112.01281203)
\closepath
\moveto(1372.00002,292.01281203)
\curveto(1372.00002,286.48996453)(1367.5228675,282.01281203)(1362.00002,282.01281203)
\curveto(1356.4771725,282.01281203)(1352.00002,286.48996453)(1352.00002,292.01281203)
\curveto(1352.00002,297.53565953)(1356.4771725,302.01281203)(1362.00002,302.01281203)
\curveto(1367.5228675,302.01281203)(1372.00002,297.53565953)(1372.00002,292.01281203)
\closepath
\moveto(1372.00002,112.01281203)
\curveto(1372.00002,106.48996453)(1367.5228675,102.01281203)(1362.00002,102.01281203)
\curveto(1356.4771725,102.01281203)(1352.00002,106.48996453)(1352.00002,112.01281203)
\curveto(1352.00002,117.53565953)(1356.4771725,122.01281203)(1362.00002,122.01281203)
\curveto(1367.5228675,122.01281203)(1372.00002,117.53565953)(1372.00002,112.01281203)
\closepath
}
}
\end{pspicture}
\caption{\label{pic5}}
\end{figure}
Figure \ref{pic5} gives an example where it is necessary for all but one non-split component of $L$ to be moved. Here $L$ is the three-component unlink, and the Seifert surface $R'$ is a surface of revolution about the axis shown. Figure \ref{pic6} shows that it is possible that every component of $R'$ meets every $S_i$.
\begin{figure}[htbp]
\centering
\psset{xunit=.21pt,yunit=.21pt,runit=.21pt}
\begin{pspicture}(1644,404.00219727)
{
\pscustom[linestyle=none,fillstyle=solid,fillcolor=lightgray]
{
\newpath
\moveto(360,200.00000203)
\curveto(360,133.72583205)(306.27416998,80.00000203)(240,80.00000203)
\curveto(173.72583002,80.00000203)(120,133.72583205)(120,200.00000203)
\curveto(120,266.27417201)(173.72583002,320.00000203)(240,320.00000203)
\curveto(306.27416998,320.00000203)(360,266.27417201)(360,200.00000203)
\closepath
\moveto(921,201.01281203)
\curveto(921,134.73864205)(867.27416998,81.01281203)(801,81.01281203)
\curveto(734.72583002,81.01281203)(681,134.73864205)(681,201.01281203)
\curveto(681,267.28698201)(734.72583002,321.01281203)(801,321.01281203)
\curveto(867.27416998,321.01281203)(921,267.28698201)(921,201.01281203)
\closepath
\moveto(1481,201.01281203)
\curveto(1481,134.73864205)(1427.27416998,81.01281203)(1361,81.01281203)
\curveto(1294.72583002,81.01281203)(1241,134.73864205)(1241,201.01281203)
\curveto(1241,267.28698201)(1294.72583002,321.01281203)(1361,321.01281203)
\curveto(1427.27416998,321.01281203)(1481,267.28698201)(1481,201.01281203)
\closepath
}
}
{
\pscustom[linewidth=4,linecolor=gray,linestyle=dashed,dash=12 12]
{
\newpath
\moveto(800,109.99999727)
\lineto(185,49.99999727)
\lineto(185,164.99999727)
\lineto(800,189.99999727)
\lineto(1390.00002,139.99999727)
\lineto(1390.00002,79.99999727)
\lineto(800,159.99999727)
\lineto(225,124.99999727)
\lineto(225,89.99999727)
\lineto(800,144.99999727)
\lineto(1360,59.99999727)
\lineto(1490.00002,59.99999727)
\lineto(1490.00002,339.99999727)
\lineto(1360,339.99999727)
\lineto(799.99996,254.99999727)
\lineto(224.99994,309.99999727)
\lineto(224.99994,274.99999727)
\lineto(799.99996,239.99999727)
\lineto(1390.00002,319.99999727)
\lineto(1390.00002,259.99999727)
\lineto(799.99996,209.99999727)
\lineto(184.99994,234.99999727)
\lineto(184.99994,349.99999727)
\lineto(799.99996,289.99999727)
}
}
{
\pscustom[linewidth=4,linecolor=gray]
{
\newpath
\moveto(240,289.99999727)
\lineto(800,249.99999727)
\lineto(1400.00002,329.99999727)
\lineto(1400.00002,249.99999727)
\lineto(800,204.99999727)
\lineto(175,224.99999727)
\lineto(175,359.99999727)
\lineto(800,304.99999727)
\lineto(815,304.99999727)
\lineto(815,274.99999727)
\lineto(800,274.99999727)
\lineto(195,339.99999727)
\lineto(195,244.99999727)
\lineto(800,219.99999727)
\lineto(1380.00002,269.99999727)
\lineto(1380.00002,309.99999727)
\lineto(800,234.99999727)
\lineto(215,264.99999727)
\lineto(215,319.99999727)
\lineto(800,264.99999727)
\lineto(1360,349.99999727)
\lineto(1500.00002,349.99999727)
\lineto(1500.00002,49.99999727)
\lineto(1360,49.99999727)
\lineto(800,134.99999727)
\lineto(215,79.99999727)
\lineto(215,134.99999727)
\lineto(800,164.99999727)
\lineto(1380.00002,89.99999727)
\lineto(1380.00002,129.99999727)
\lineto(800,179.99999727)
\lineto(195,154.99999727)
\lineto(195,59.99999727)
\lineto(800,124.99999727)
\lineto(815,124.99999727)
\lineto(815,94.99999727)
\lineto(800,94.99999727)
\lineto(175,39.99999727)
\lineto(175,174.99999727)
\lineto(800,194.99999727)
\lineto(1400.00002,149.99999727)
\lineto(1400.00002,69.99999727)
\lineto(800,149.99999727)
\lineto(240,109.99999727)
}
}
{
\pscustom[linewidth=4,linecolor=gray,linestyle=dashed,dash=4 4]
{
\newpath
\moveto(1359.99998,289.99999727)
\lineto(799.99996,224.99999727)
\lineto(204.99994,254.99999727)
\lineto(204.99994,329.99999727)
\lineto(799.99996,269.99999727)
\lineto(989.99996,299.99999727)
\lineto(1359.99998,359.99999727)
\lineto(1510.00002,359.99999727)
\lineto(1510.00002,39.99999727)
\lineto(1360,39.99999727)
\lineto(990,99.99999727)
\lineto(800,129.99999727)
\lineto(205,69.99999727)
\lineto(205,144.99999727)
\lineto(800,174.99999727)
\lineto(1360,109.99999727)
}
}
{
\pscustom[linewidth=2,linecolor=black]
{
\newpath
\moveto(40,199.99999727)
\lineto(1601.00002,200.99999727)
}
}
{
\pscustom[linewidth=4,linecolor=black,fillstyle=solid,fillcolor=black]
{
\newpath
\moveto(250,290.00000203)
\curveto(250,284.47715453)(245.5228475,280.00000203)(240,280.00000203)
\curveto(234.4771525,280.00000203)(230,284.47715453)(230,290.00000203)
\curveto(230,295.52284953)(234.4771525,300.00000203)(240,300.00000203)
\curveto(245.5228475,300.00000203)(250,295.52284953)(250,290.00000203)
\closepath
\moveto(250,110.00000203)
\curveto(250,104.47715453)(245.5228475,100.00000203)(240,100.00000203)
\curveto(234.4771525,100.00000203)(230,104.47715453)(230,110.00000203)
\curveto(230,115.52284953)(234.4771525,120.00000203)(240,120.00000203)
\curveto(245.5228475,120.00000203)(250,115.52284953)(250,110.00000203)
\closepath
\moveto(810,290.00000203)
\curveto(810,284.47715453)(805.5228475,280.00000203)(800,280.00000203)
\curveto(794.4771525,280.00000203)(790,284.47715453)(790,290.00000203)
\curveto(790,295.52284953)(794.4771525,300.00000203)(800,300.00000203)
\curveto(805.5228475,300.00000203)(810,295.52284953)(810,290.00000203)
\closepath
\moveto(810,110.00000203)
\curveto(810,104.47715453)(805.5228475,100.00000203)(800,100.00000203)
\curveto(794.4771525,100.00000203)(790,104.47715453)(790,110.00000203)
\curveto(790,115.52284953)(794.4771525,120.00000203)(800,120.00000203)
\curveto(805.5228475,120.00000203)(810,115.52284953)(810,110.00000203)
\closepath
\moveto(1370,290)
\curveto(1370,284.4771525)(1365.5228475,280)(1360,280)
\curveto(1354.4771525,280)(1350,284.4771525)(1350,290)
\curveto(1350,295.5228475)(1354.4771525,300)(1360,300)
\curveto(1365.5228475,300)(1370,295.5228475)(1370,290)
\closepath
\moveto(1370,110.00000203)
\curveto(1370,104.47715453)(1365.5228475,100.00000203)(1360,100.00000203)
\curveto(1354.4771525,100.00000203)(1350,104.47715453)(1350,110.00000203)
\curveto(1350,115.52284953)(1354.4771525,120.00000203)(1360,120.00000203)
\curveto(1365.5228475,120.00000203)(1370,115.52284953)(1370,110.00000203)
\closepath
}
}
\end{pspicture}
\caption{\label{pic6}}
\end{figure}
\end{example}
\subsection{The case of two non-split components}

Suppose $N=2$. That is, suppose that $L$ has exactly two non-split components. Then we may discard $S_2$ and take $S=S_1$.
By Proposition \ref{movehalfprop}, $R'$ can be isotoped to be disjoint from $S$.
This means that we may define a bijection $\pi_S\colon\V(\ms(L))\to\V(\ms_p(L_1))\times\V(\ms_p(L_2))$ by dividing $\Sphere$ along $S$.

\begin{lemma}\label{fibredgivesisomorphismlemma}
If $L_2$ is fibred as a link in $\Sphere$, then $\ms(L)\cong\ms_p(L_1)$.
\end{lemma}
\begin{proof}
There exists a taut Seifert surface $R_2$  for $L_2$ such that $\V(\ms_p(L_2))=\{[R_2]\}$. 
The map $[R_1]\mapsto [R_1\cup R_2]$ gives a bijection from $\V(\ms_p(L_1))$ to $\V(\ms(L))$.
It is also clear that $[R_1]$ and $[R'_1]$ are adjacent in $\ms_p(L_1)$ if and only if $[R_1\cup R_2]$ and $[R'_1\cup R_2]$ are adjacent in $\ms(L)$. 
\end{proof}

We now assume $L_1$ and $L_2$ are not fibred.
For $i=1,2$, let $M_i$ be the closure of the component of $M\setminus S=\Sphere\setminus(\nhd(L)\cup S)$ that meets $\partial\!\nhd(L_i)$, and let $\widetilde{M}_i$ be the infinite cyclic cover of $M_i$.

\begin{lemma}\label{onesideequallemma}
For $i=1,2$, let $R_i$ and $R'_i$ be taut Seifert surfaces for $L_i$ disjoint from $S$. Suppose that $[R_1]=[R'_1]$ in $\ms_p(L_1)$. Then $\pi_S^{-1}([R_1],[R_2])$ is adjacent to $\pi_S^{-1}([R'_1],[R'_2])$ in $\ms(L)$ if and only if $[R_2]$ is adjacent to $[R'_2]$ in $\ms_p(L_2)$.
\end{lemma}

\begin{lemma}\label{spheretoadjacencylemma}
For $i=1,2$, let $R_i$ and $R'_i$ be taut Seifert surfaces for $L_i$ disjoint from $S$ such that $[R_i]$ and $[R'_i]$ are adjacent in $\ms_p(L_i)$.
Suppose also that $R_i$ and $R'_i$ are tight and $S$ lies in a component of $M_i\setminus(R_i\cup R'_i)$ above $R_i$ and below $R'_i$ for $i=1,2$.
Then $\pi_S^{-1}([R_1],[R_2])$ and $\pi_S^{-1}([R'_1],[R'_2])$ are adjacent in $\ms(L)$. 
\end{lemma}
\begin{proof}
Let $V$ be a lift of $M\setminus(R_1\cup R_2)$ to $\widetilde{M}$, and let $\widetilde{S}$ be the lift of $S$ in $V$.
Let $V'$ be the lift of $M\setminus(R'_1\cup R'_2)$ that contains $\widetilde{S}$. 
For $i=1,2$, let $V_i$ be the lift of $M_i\setminus R_i$ contained in $V$, and let $V'_i$ be the lift of $M_i\setminus R'_i$ contained in $V'$.
Then $V=V_1\cup V_2$ and $V'=V'_1\cup V'_2$.
Note that $R_i$ and $R'_i$ do not coincide and are tight, so $V'_i$ meets $\tau^m(V_i)$ for exactly two integers $m$. 
As $S$ lies in a component of $M_i\setminus(R_i\cup R'_i)$ above $R_i$ and below $R'_i$,  we know that $V'_i$ meets $V_i$ and $\tau^{-1}(V_i)$. 
Hence $V'$ meets $\tau^m(V)$ only when $m\in\{0,-1\}$. This shows that $R_1\cup R_2$ and $R'_1\cup R'_2$ are tight.
\end{proof}

\begin{proposition}\label{sphereinsamecomponentprop}
For $i=1,2$, let $R_i$ and $R'_i$ be taut Seifert surfaces for $L_i$ disjoint from $S$ such that $\pi_S^{-1}([R_1],[R_2])$ and $\pi_S^{-1}([R'_1],[R'_2])$ are adjacent in $\ms(L)$. 
Assume also that $[R_i]\neq[R'_i]$ in $\ms_p(L_i)$ for $i=1,2$.
Then $[R_i]$ and $[R'_i]$ are adjacent in $\ms_p(L_i)$ for $i=1,2$.

Suppose further that, for $i=1,2$, $R_i$ and $R'_i$ are tight. Then either $S$ lies in a component of $M_i\setminus(R_i\cup R'_i)$ above $R_i$ and below $R'_i$ for $i=1,2$, or $S$ lies in a component of $M_i\setminus(R_i\cup R'_i)$ below $R_i$ and above $R'_i$ for $i=1,2$.
\end{proposition}
\begin{proof}
Let $R''_1$ and $R''_2$ be copies of $R'_1$ and $R'_2$ respectively.
Isotope $R''_1$ and $R''_2$ in the complement of $S$ so that $R_1\cup R_2$ and $R''_1\cup R''_2$ are tight as Seifert surfaces for $L$ in $M$.
Then $R_i$ and $R''_i$ are tight in $M_i$ for $i=1,2$.
This completes the first part of the lemma.

Without loss of generality, assume that $S$ lies in a component of $M_1\setminus(R_1\cup R''_1)$ above $R_1$ and below $R''_1$.
Let $V$ be a lift of $M\setminus R$ to $\widetilde{M}$. 
Let $V''$ be the lift of $M\setminus R''$ that meets $V$ and $\tau(V)$.
Also, let $\widetilde{S}$ be the lift of $S$ contained in $V''$. 
By considering $\widetilde{M}_1$ we see that $\widetilde{S}$ lies in $\tau(V)$.
Thus $S$ also lies in a component of $M_2\setminus(R_2\cup R''_2)$ above $R_2$ and below $R''_2$.

Suppose now that $R_i$ and $R'_i$ are tight for $i=1,2$.

Let $\widetilde{R}''_i$ be the lift of $R''_i$ in $\tau(V)$, which lies between $V''$ and $\tau(V'')$.
The isotopy from $R''_i$ to $R'_i$ lifts to an isotopy from $\widetilde{R}''_i$ to a lift $\widetilde{R}'_i$ of $R'_i$.
Since the isotopy of $R''_i$ is disjoint from $S$, $\widetilde{R}'_i$ lies above $\widetilde{S}$ and below $\tau(\widetilde{S})$.
As $R'_i$ does not cross $R_i$, this means $\widetilde{R}'_i$ is either in $\tau(V)$ or in $\tau^2(V)$.

If $\widetilde{R}'_i$ is in $\tau^2(V)$ then it is separated from $\widetilde{R}''_i$ by the lift $\widetilde{R}_i$ of $R_i$ that lies between $\tau(V)$ and $\tau^2(V)$.
Let $M_i'=\Sphere\setminus\nhd(L_i)$, and let $\widetilde{M}_i'$ be the infinite cyclic cover of $M_i'$.
Note that $M_i\subset M\subset M_i'$ and $\widetilde{M}_i\subset\widetilde{M}\subset\widetilde{M}_i'$.
Because $\widetilde{R}''_i$ and $\widetilde{R}'_i$ are isotopic in $\widetilde{M}_i'$, by Theorem \ref{productregionthm} there is a product region $T$ between them in $\widetilde{M}_i'$. This product region does not contain any lift of $S$, so $T\subset\widetilde{M}_i$, but $T$ does contain $\widetilde{R}_i$. Hence $\widetilde{R}_i$ is isotopic to $\widetilde{R}'_i$ in $\widetilde{M}_1$, contradicting that $R_i$ is not isotopic to $R'_i$ in $M_1$.

Therefore $\widetilde{R}'_i$ is in $\tau(V)$.
As $\widetilde{R}'_i$ is above $\widetilde{S}$, this means $S$ lies in a component of $M_i\setminus(R_i\cup R''_i)$ above $R_i$ and below $R''_i$.
\end{proof}

\begin{definition}[\cite{MR0050886} Chapter II Definition 8.7]
A simplicial complex $\mathcal{X}$ is \textit{ordered} if there is a binary relation $\leq$ on the vertices of $\mathcal{X}$ with the following properties.
\begin{itemize}
\item $(u\leq v\textrm{ and }v\leq u)\Rightarrow u=v$.
\item If $u,v$ are distinct, $(u\leq v\textrm{ or }v\leq u)\Leftrightarrow u\textrm{ and }v $ are adjacent.
\item If $u,v,w$ are vertices of a 2--simplex then $(u\leq v\textrm{ and }v\leq w)\Rightarrow u\leq w$.
\end{itemize}
\end{definition}

\begin{definition}[\cite{MR0050886} Chapter II Definition 8.8]
Let $\mathcal{X}_1,\mathcal{X}_2$ be ordered simplicial complexes. We define the simplicial complex $\mathcal{X}_1\times\mathcal{X}_2$. Its vertices are given by the set $\V(\mathcal{X}_1)\times\V(\mathcal{X}_2)$. Vertices $(u_0, v_0),\ldots, (u_n, v_n)$ span an $n$--simplex if the following hold.
\begin{itemize}
\item $\{u_0,\ldots,u_n\}$ is an $m$--simplex of $\mathcal{X}_1$ for some $m\leq n$.
\item $\{v_0,\ldots,v_n\}$ is an $m$--simplex of $\mathcal{X}_2$ for some $m\leq n$.
\item The relation defined by $(u,v) \leq (u',v') \Leftrightarrow (u\leq u'\textrm{ and }v \leq v')$ gives a total linear order on $(u_0, v_0),\ldots, (u_n, v_n)$.
\end{itemize}
\end{definition}

\begin{proposition}[\cite{MR0050886} Chapter II Lemma 8.9]\label{productorderedcomplexprop}
The map $|\mathcal{X}_1\times\mathcal{X}_2| \to |\mathcal{X}_1|\times|\mathcal{X}_2|$ induced by projection is a homeomorphism.
\end{proposition}

\begin{definition}\label{orderdefn}
For $i=1,2$, given taut Seifert surfaces $R_i$ and $R'_i$ for $L_i$ in $M_i$ such that $[R_i]$ and $[R'_i]$ are adjacent in $\ms_p(L_i)$, isotope them to be tight in $M_i$, and
say $R_i>_i R'_i$ if $S$ lies in a component of $M_i\setminus(R_i\cup R'_i)$ above $R_i$ and below $R'_i$.

For $i=1,2$ define $\leq_i$ on $\V(\ms_p(L_i))$ by $[R_i]\leq_i[R'_i]$ if and only if either $[R_i]=[R'_i]$ or $R'_i>_i R_i$.
\end{definition}

\begin{lemma}
$(\ms_p(L_i),\leq_i)$ is an ordered simplicial complex.
\end{lemma}
\begin{proof}
It follows from the proof of Proposition \ref{sphereinsamecomponentprop} that Definition \ref{orderdefn} gives a well-defined binary relation on $\V(\ms_p(L_i))$, and that if $[R_i]\leq_i[R'_i]\leq_i[R_i]$ then $[R_i]=[R'_i]$.
For $[R_i]\neq[R'_i]$, the relation $\leq_i$ is defined precisely when $[R_i]$ and $[R'_i]$ are adjacent in $\ms_p(L_i)$.

Suppose $[R_i], [R'_i],[R''_i]$ are the vertices of a 2--simplex in $\ms_p(L_i)$, and that $[R_i]\leq_i[R'_i]\leq[R''_i]$.
Without loss of generality, $R_i$ and $R'_i$ are tight, as are $R'_i$ and $R''_i$.
Let $\widetilde{S}$ be a lift of $S$ to $\widetilde{M}_i$, and
let $V_i$, $V'_i$ and $V''_i$ be the lifts to $\widetilde{M}_i$ of $M_i\setminus R_i$, $M_i\setminus R'_i$ and $M_i\setminus R''_i$ respectively that meet $\widetilde{S}$.
Then $V'_i$ meets $V_i$ and $\tau(V_i)$, while $V''_i$ meets $V'_i$ and $\tau(V'_i)$.

It may be that $R_i$ and $R''_i$ are not tight, but they can be isotoped to be tight.
Let $\widetilde{R}''_i$ be the lift of $R''_i$ such that $V''_i$ lies between $\widetilde{R}''_i$ and $\tau(\widetilde{R}''_i)$.
If $[R''_i]\leq_i[R_i]$ then $\widetilde{R}''_i$ can be isotoped in $\widetilde{M}_i$ into $\tau^{-1}(V_i)$. This contradicts that $R''_i$ is not isotopic to $R_i$ and that $R''_i$ is not isotopic to $R'_i$.
Hence, instead, $[R_i]\leq_i[R''_i]$.
\end{proof}

\begin{theorem}\label{productcomplexthm}
$|\ms(L)|\cong|\ms_p(L_1)|\times|\ms_p(L_2)|$.
\end{theorem}
\begin{proof}
By Proposition \ref{productorderedcomplexprop}, it suffices to show that $\pi_S^{-1}([R_1],[R_2])$ is adjacent to $\pi_S^{-1}([R'_1],[R'_2])$ in $\ms(L)$ if and only if either $[R_i]\leq_i[R'_i]$ for $i=1,2$ or $[R'_i]\leq_i[R_i]$ for $i=1,2$.

First suppose $[R_1]=[R'_1]$ and $[R_2]\neq[R'_2]$.
By Lemma \ref{onesideequallemma}, $\pi_S^{-1}([R_1],[R_2])$ is adjacent to $\pi_S^{-1}([R'_1],[R'_2])$ in $\ms(L)$ if and only if $[R_2]$ is adjacent to $[R'_2]$ in $\ms_p(L_2)$, which is the case if and only if either $[R_2]\leq_2[R'_2]$ or $[R'_2]\leq_2[R_2]$.

Suppose instead that $[R_1]\neq[R'_1]$ and $[R_2]\neq[R'_2]$.
Since the partial orders $\leq_1$ and $\leq_2$ are well defined, we may assume that $R_1$ and $R'_1$ are tight, as are $R_2$ and $R'_2$.
The result now follows from Lemma \ref{spheretoadjacencylemma}, Proposition \ref{sphereinsamecomponentprop} and Definition \ref{orderdefn}.
\end{proof}

\begin{remark}
If $L_1$ is fibred then $\ms(L)$ is a single vertex, so $|\ms_p(L_1)|$ is a point by Lemma \ref{fibredgivesbijectionlemma}.
Hence Lemma \ref{fibredgivesisomorphismlemma} shows that
Theorem \ref{productcomplexthm} also holds if either $L_1$ or $L_2$ is fibred.
\end{remark}


\appendix
\section{Making surfaces simultaneously tight}

Assume $L$ is non-split and non-fibred. 
We will see that surfaces in pairwise adjacent vertices of $\ms(L)$ can be made simultaneously pairwise disjoint and tight. Moreover, once this has been done the relative positions of the surfaces are essentially unique.

\begin{lemma}[\cite{2011arXiv1109.0965B} Lemma 4.11]\label{arcintersectionlemma}
Let $R_1, R_2$ be fixed, almost disjoint, taut Seifert surfaces for $L$. Then $R_1,R_2$ demonstrate that their isotopy classes are at most distance 1 apart in $\ms(L)$ if and only if the following holds for all pairs $(x,y)$ of points of $R_1\setminus R_2$. 

Choose a product neighbourhood $\nhd(R_1)=R_1\times[-1,1]\subset \Sphere\setminus\nhd(L)$ for $R_1$ with $R_1=R_1\times\{0\}$, such that $R_1\times\{1\}$ lies on the positive side of $R_1$. Let $\rho\colon I\to M$ be any path with $\rho(0)=(x,1)$ and $\rho(1)=(y,-1)$ that is disjoint from $R_1$ and transverse to $R_2$. 
Then the algebraic intersection number of $\rho$ and $R_2$ is 1.
\end{lemma}

\begin{proposition}\label{alltightprop}
Let $[R_0],\ldots,[R_n]$ be the distinct vertices of a simplex of $\ms(L)$, and let $[S]$ be any vertex.
Then the surfaces $R_0,\ldots,R_n,S$ can be positioned such that $R_i,R_j$ are disjoint and tight for $i\neq j$, and such that $R_i,S$ are almost transverse with simplified intersection for each $i$, except that any component of $R_i$ that is parallel to a component of $S$ lies below that of $S$ rather than coinciding with it. 
\end{proposition}
\begin{proof}
We proceed by induction on $n$. The surface $R_0$ can be made almost transverse to $S$ with simplified intersection. Moving any component of $R_0$ that coincides with one of $S$ downwards completes the base case.

Suppose that the surfaces $R_0,\ldots,R_m$ have been positioned as required relative to each other and to $S$. We now position $R_{m+1}$ by induction. By Lemma \ref{productregionsidelemma}, $R_{m+1}$ can be made disjoint from and tight relative to $R_0$. Suppose $R_{m+1}$ has been positioned so that it is disjoint from and tight relative to each of the surfaces $R_0,\ldots,R_k$ for some $k<m$. By a small isotopy make $R_{m+1}$ transverse to $R_{k+1}$. Suppose $R_{m+1}\cap R_{k+1}\neq\emptyset$. Temporarily delete all components of either surface that are disjoint from the other surface. Since these two surfaces can be made disjoint, by Theorem \ref{productregionthm} there is a product region $T_0$ between them. In addition, $\partial T_0$ meets $R_{m+1}\cap R_{k+1}$.
If $T_0$ is disjoint from the surfaces $R_0,\ldots,R_k$ it can be used to simplify $R_{m+1}\cap R_{k+1}$ without affecting the interaction of the other surfaces. Suppose $T_0$ meets $R_i$ for some $i\leq k$. As $R_i$ is disjoint from $R_{k+1}$ and from $R_{m+1}$, Proposition \ref{surfaceinproductprop} gives that $R_i\cap T_0$ is parallel to the horizontal boundary of $T_0$. This contradicts that $\partial T_0$ meets $R_{m+1}\cap R_{k+1}$. Thus we can make $R_{m+1}$ disjoint from $R_{k+1}$. 

By Lemma \ref{isotopeincomplementlemma}, 
$R_{m+1},R_{k+1}$ are tight unless components of $R_{m+1}$ and $R_{k+1}$ are parallel to each other. Consider a product region $T_1$ between two such components. If $T_1$ is disjoint from $R_0,\ldots,R_k$ then we can use it to switch the components of $R_{m+1}$ and $R_{k+1}$ if necessary as needed to make them tight. If $T_1$ meets a component of $R_i$ for some $i\leq k$ then this component is also parallel to those of $R_{m+1},R_{k+1}$. In this case we must show that the component of $R_{m+1}$ is already correctly positioned relative to that of $R_{k+1}$ to be tight.

Let $R_{j,0}$ be the component of $R_j$ in question for $j=i,k+1,m+1$. Without loss of generality, $T_1$ lies below $R_{k+1,0}$ and above $R_{m+1,0}$. Some component $R_{m+1,1}$ of $R_{m+1}$ is not parallel to any component of $R_{k+1}$. Let $x\in R_{m+1,1}\setminus R_{k+1}$ and $y\in R_{m+1,0}$. Choose a path $\rho$ from just above $x$ to just below $y$ that is disjoint from $R_{m+1}$ as in Lemma \ref{arcintersectionlemma}. Since $R_i,R_{m+1}$ are tight, the algebraic intersection of $\rho$ and $R_{i}$ is 1. Let $x_1$ be the final point of $\rho\cap R_{i}$ along $\rho$. 

Suppose that $\rho$ passes through $R_{i}$ in the negative direction at $x_1$. Then there is a section of $\rho$ that passes through $R_{i}$ exactly twice, both times in the positive direction. Since this section of $\rho$ is disjoint from $R_{m+1}$, this contradicts that $R_i$ and $R_{m+1}$ are tight. Thus $\rho$ passes through $R_{i}$ in the positive direction at $x_1$. 

Let $\rho'$ be the path that runs along $\rho$ from $x_1$ to $y$ and then continues through the product region $T_1$ to a point on $R_{i,0}$.
Then $\rho'$ has algebraic intersection 1 with $R_{k+1}$. Let $x_2$ be the last point of $\rho'\cap R_{k+1}$ along $\rho'$. Note that $\rho'$ does not meet $R_{k+1}$ in $T_1$, so $x_2$ occurs before $y$ on $\rho'$.

Let $\rho''$ be the path that runs along $\rho'$ from $x_2$ to where it meets $R_i$ in $T_1$ and then continues through $T_1$ to $R_{k+1}$. Then $\rho''$ has algebraic intersection 1 with $R_{m+1}$ as they only meet at $y$. This shows that $R_{m+1,0}$ is positioned correctly with respect to $R_{k+1}$. 

This completes the induction to show that $R_{m+1}$ can be positioned correctly relative to $R_0,\ldots,R_m$, which in turn completes the induction to show that $R_0,\ldots,R_n$ can be suitably positioned relative to each other. It remains to show that $S$ can be put into the required position. 

By a small isotopy, $S$ can be made transverse to the surfaces $R_0,\ldots,R_n$. Suppose $S$ does not have simplified intersection with these surfaces. Then there is a product region $T_2$ between $S$ and $R_i$ for some $i\leq n$. If $T_2$ meets $S\cap R_i$ then isotoping $S$ across $T_2$ reduces the number of intersection curves of $S$ with $\bigcup R_j$. We may therefore remove all such product regions. It remains to consider those components of $S$ that are parallel to components of $\bigcup R_j$. Suppose a component $S_0$ of $S$ is parallel to a component $R_{i,0}$ of $R_i$. Consider all components of $\bigcup R_j$ that are parallel to $S_0$. If all such components lie below $S_0$ then $S_0$ is correctly positioned. Else, let $T_3$ be the product region between the top side of $S_0$ and the bottom side of the parallel component furthest above $S_0$. Note that such a component is well defined as the components of $\bigcup R_j$ are disjoint. By Proposition \ref{surfaceinproductprop} and our previous positioning of $S$ we see that all surfaces that meet $T_3$ are parallel to $S_0$. Thus moving $S_0$ to above all parallel components does not change its simplified intersection elsewhere.
\end{proof}

\begin{lemma}\label{alltightbdylemma}
Suppose that taut Seifert surfaces $R_0,\ldots,R_n$ have been chosen with $\partial R_0=\partial R_1=\cdots=\partial R_n$ such that, for $i\neq j$, $R_i$ and $R_j$ can be isotoped to be tight keeping their boundaries fixed. Then they can be isotoped keeping their boundaries fixed to be pairwise tight and disjoint except along their boundaries.
\end{lemma}
\begin{proof}
This can be seen by replacing Theorem \ref{productregionthm} with Corollary \ref{isotopyrelbdycor2} in the proof of Proposition \ref{alltightprop}.
\end{proof}

\begin{lemma}\label{orderinglemma}
Let $R,R',S$ be as in Definition \ref{orderingdefn}, except that any component of $R$ or $R'$ that is parallel to one of $S$ may lie below it instead of coinciding with it, and parallel components of $R,R'$ need not coincide. Suppose additionally that $R,R'$ are tight. Then lifts of the complements of these surfaces can be used to decide whether $R'<_{S}R$ as in Definition \ref{orderingdefn}.
\end{lemma}
\begin{proof}
Moving some (or even all) components of $R$ downwards does not change the choice of lift $V_{R}$ of $M\setminus R$. Since $R,R'$ are tight, $V_{R}\subset V_{R'}\cup \tau(V_{R'})\cup\tilde{R}'$, where $\tilde{R}'$ is the lift of $R'$ that lies between $V_{R'}$ and $\tau(V_{R'})$. In particular, $\tilde{R}'$ lies within the closure of $V_{R}$, as it would in Definition \ref{orderingdefn}. This means the same lift $V_{R'}$ of $M\setminus R'$ has been chosen. Finally, moving components of $R'$ downwards does not change whether $V_{R'}$ meets $V_{S}$. 
\end{proof}

\begin{proposition}\label{surfacesinorderprop}
Let $[R_0],[R_1],[R_2]$ be the vertices of a 2--simplex in $\ms(L)$, with $[R_0]<_{R_p}[R_1]<_{R_p}[R_2]$.
Suppose that $R_0,R_1,R_2$ are pairwise disjoint and tight.
Then on any component of $\partial M$ the boundaries of the surfaces occur in the order $\partial R_0,\partial R_1,\partial R_2,\partial R_0$ as measured in the positive direction around $L$.
\end{proposition}
\begin{proof}
Let $R'_0,R'_1,R'_2$ be copies of $R_0,R_1,R_2$ respectively, and
position the surfaces $R'_0,R'_1,R'_2,R_p$ as in Proposition \ref{alltightprop}.
By Lemma \ref{orderinglemma}, we may use $R'_0,R'_1,R'_2,R_p$ in Definition \ref{orderingdefn}.

Let $V_p$ be a lift of $M\setminus R_p$ to $\widetilde{M}$.
Let $V_2$ be the lift of $M\setminus R'_2$ such that $V_2\cap V_p\neq \emptyset$ but $V_2\cap\tau(V_p)=\emptyset$, and let $V_1$ be the lift of $M\setminus R'_1$ such that $V_1\cap V_2\neq\emptyset$ but $V_1\cap\tau(V_2)=\emptyset$.
Since $[R_1]<_{R_p}[R_2]$ we know that $V_1\cap V_p\neq\emptyset$. 
Let $V_0$ be the lift of $M\setminus R'_0$ such that $V_0\cap V_1\neq\emptyset$ but $\V_0\cap \tau(V_1)=\emptyset$.
Then $V_0\cap V_p\neq\emptyset$, as $[R_0]<_{R_p}[R_1]$.
From this it follows that on $\partial M$ the boundaries occur in the order $\partial R_0, \partial R_1, \partial R_2,\partial R_0$.

Now we compare $R'_0,R'_1,R'_2$ to $R_0,R_1,R_2$.
We may assume that $R'_0$ coincides with $R_0$.
By Lemma \ref{isotopeincomplementlemma}, $R'_1$ is isotopic to $R_1$ in the complement of $R_0$. Thus we may also assume that $R'_1$ coincides with $R_1$.
Finally, by Lemma \ref{isotopeincomplementlemma} again we find that $R'_2$ is isotopic to $R_2$ in the complement of $R_0\cup R_1$.
\end{proof}

\begin{remark}
Proposition \ref{alltightprop} and Proposition \ref{surfacesinorderprop} together prove Theorem \ref{uniquepositionthm}.
\end{remark}


\bibliography{splitlinksreferences}
\bibliographystyle{hplain}

\bigskip
\noindent
CIRGET, D\'{e}partement de math\'{e}matiques

\noindent
UQAM

\noindent
Case postale 8888, Centre-ville

\noindent
Montréal, H3C 3P8

\noindent
Quebec, Canada

\smallskip
\noindent
\textit{jessica.banks[at]lmh.oxon.org}

\end{document}